\documentclass[a4paper, 11pt]{amsart}

\usepackage[T2A,T1]{fontenc}  
\usepackage[utf8]{inputenc}

\usepackage[russian, french, british]{babel}

\usepackage{amssymb, amsmath}
\usepackage{mathtools}
\usepackage{bm}

\usepackage[cal=boondoxo]{mathalfa}

\usepackage{comment}
\usepackage{url}

\usepackage[hidelinks]{hyperref}

\newtheorem{THEOREM}{Theorem}[section]

\newtheorem{theorem}[THEOREM]{Theorem}
\newtheorem{dtheorem}[THEOREM]{Wishful Theorem}
\newtheorem{lemma}[THEOREM]{Lemma}
\theoremstyle{definition}
\newtheorem{definition}[THEOREM]{Definition}

\newtheorem{corollary}[THEOREM]{Corollary}
\newtheorem{example}[THEOREM]{Example}
\newtheorem{fact}[THEOREM]{Fact}
\newtheorem{remark}[THEOREM]{Remark}
\newtheorem{Observation}[THEOREM]{Observation}
\newenvironment{observation}{\begin{Observation}}
{\end{Observation}}
\newtheorem{Notation}[THEOREM]{Notation}
\newenvironment{notation}{\begin{Notation}}
{\end{Notation}}

\newcommand{\forces}{\Vdash}
\newcommand{\dom}{\text{dom}}

\newcommand{\supt}{\text{supt}}

\newcommand{\ran}{\text{ran}}

\newcommand{\kpo}{\kappa\mathbf{PO}}
\newcommand{\isz}{\mathbf{IS}\zeta}

\newcount\skewfactor
\def\mathunderaccent#1#2 {\let\theaccent#1\skewfactor#2
\mathpalette\putaccentunder}
\def\putaccentunder#1#2{\oalign{$#1#2$\crcr\hidewidth
\vbox to.2ex{\hbox{$#1\skew\skewfactor\theaccent{}$}\vss}\hidewidth}}
\def\name{\mathunderaccent\tilde-3 }

\newcommand{\rest}{\upharpoonright}

\newcommand{\cf}{\text{cf}}

\newcommand{\eop}{\bigstar}
\newcommand{\DD}{\mathcal D}

\newcommand{\FF}{\mathcal F}
\newcommand{\HH}{\mathcal H}

\newcommand{\eod}{\scriptstyle\blacksquare}

\title{Iterating Generalised Perfect Set Forcing along well-founded orders}

\author{Mirna D{\v z}amonja, by-name Logique Consult, Liège, Belgium}

\date{}

\begin{document}

\maketitle

\begin{abstract} The technique of geometric forcing iteration was developed by Kanovei
\cite{zbMATH01335192} and used to prove that Perfect Set Forcing can be iterated with countable supports along any partial order, while preserving $\aleph_1$. In \cite{Property-B} we considered Generalised Perfect Set Forcing ${\mathbb P} (\mathcal F)$, defined with respect to a filter $\FF$ on a cardinal $\kappa$ satisfying  $\kappa^{<\kappa}=\kappa$. We proved that for appropriate filters, the iteration of ${\mathbb P} (\mathcal F)$ with supports of size
$\le\kappa$, along any ordinal, preserves cardinals up to and including $\kappa^+$.

In this paper we show that there is a version of the geometric iteration technique that applies to ${\mathbb P} (\mathcal F)$. It yields that for $\kappa$ satisfying  $\kappa^{<\kappa}=\kappa$ and for appropriate filters $\mathcal F$, the forcing  ${\mathbb P} (\FF)$ can be iterated with supports of size $\le\kappa$ along any well-founded partial order, while
preserving cardinals up to and including $\kappa^+$. As an application of our technique we obtain that some common notions of arboreal forcings on $\omega$ can be iterated with countable supports along any well-founded partial order and that such iterations preserve $\aleph_1$.

\end{abstract}

\section{Introduction and Motivation} Perfect Set Forcing is a forcing notion whose conditions are perfect subtrees of $2^{<\omega}$ (or, more generally, of $\omega^{<\omega}$). A perfect tree is a nonempty subtree of $2^{<\omega}$ with no terminal nodes; forcing with such a tree adds a new real (an infinite branch through the tree) that is ``generic'' in a strong topological sense. Similarly to Cohen forcing, which adds a very ``free'' generic real, or random forcing, which adds a measure-theoretically generic real, Perfect Set Forcing adds a special real: one of minimal Turing degree. This forcing notion is closely related to Albert Muchnik's solution to the Post problem in \cite{zbMATH03117565}. It was formulated as a forcing notion by Gerald Sacks in \cite{Sacks}.

For a set theorist, the importance of Perfect Set Forcing lies in its ability to increase the size of the continuum while keeping the cardinal invariants of the continuum low. 
Iterations or products of Perfect Set Forcing are often used to obtain consistency results about the size of the continuum, the existence of certain special sets, or the behaviour of various sets of reals, For example, over a model of CH, Perfect Set Forcing can be iterated with countable supports $\omega_2$ many times while preserving $\aleph_1$, to obtain what is known as Sacks model. Various combinatorial properties of this model are known, see \cite{zbMATH03664937} and \cite{GeshckeQuickert}.

One would think that the idea of Perfect Set Forcing can be naturally generalised to higher cardinals, giving rise to the notion of \emph{Perfect Set Forcing at $\kappa$}, or, more generally, \emph{$\kappa$-perfect forcing}. This actually turns out to be quite difficult, with several mistaken attempts available in the literature. The reason for this is the emerging evidence of the impossibility of generalising to uncountable cardinals the rich combinatorial and descriptive set theory that we know about $\omega$. The papers \cite{Inamdar-Rinot-nonstructure-aleph2} by Tanmay Inamdar and Assaf Rinot and \cite{FriedmanHyttinenKulikov2014} by Sy-David Friedman, Tapani Hyttinen and Vadim Kulikov describe some of these difficulties, giving not only heuristic evidence but actual proofs of the impossibility of such a generalisation. On the other hand, the very natural and elegant programme introduced by Itay Neeman in \cite{Neeman}, which uses forcing with two types of models as side conditions to redo the classical iterations, opens a door to generalisations to higher cardinals. There has therefore been a tension between the possibilities that Neeman's method offers and the discovery of ZFC limitations of what we can expect as consistency results at cardinals larger than $\aleph_1$. 

S.-D. Friedman and Victoria Gitman \cite{FriedmanGitman:ModelOfACNotDCInaccessible} introduced a version of Perfect Set Forcing for $\kappa$-inaccessible, which they called $\kappa$-Perfect Set Forcing. They proved that finite iterations of such posets are of the same kind, but that this is not the case for $\omega$-sequences of them. Instead, they considered complex systems of iterations of such posets and a generalisation of the related notion of Jensen posets along the tree $\kappa^{<\omega}$. In particular, there is no general iteration theorem available for $\kappa$-Perfect Set Forcing.

A positive result in the direction of generalising the iteration results from Perfect Set Forcing to uncountable cardinals $\kappa$ was obtained in \cite{Property-B}. There we found an analogue of Perfect Set Forcing different from $\kappa$-Perfect Set Forcing and showed that it can be iterated along the usual well-ordered iterations with supports of size $\le\kappa$, while preserving all the cardinals up and including $\kappa^+$. This requires only that $\kappa=\kappa^{<\kappa}$. This forcing is parametrised by a $(<\kappa)$-complete filter $\FF$ on $\kappa$, so we denote it 
${\mathbb P} (\mathcal F)$. In view of the iterability results which are not known for other possible generalisations of Perfect Set Forcing to uncountable cardinals, we have called the forcing introduced in \cite{Property-B} Generalised Perfect Set Forcing.

Roughly speaking, the forcing 
${\mathbb P} (\mathcal F)$ consists of perfect subtrees of ${}^{<\kappa}\kappa$ of height $\kappa$ where every splitting is into $\FF$ many points, and there are club many splittings along every branch (see Definition \ref{gSacks:def} for the precise definition). 
In \cite{Property-B} we showed that the forcing ${\mathbb P} (\mathcal F)$ has satisfactory fusion properties and preserves cardinals 
$\le\kappa^+$ when iterated with supports of size $\le\kappa$ along an ordinal. Given the difficulties and known bounds to the possible extensions of the proper forcing iteration technique which follow from \cite{Inamdar-Rinot-nonstructure-aleph2}, this shows that Generalised Perfect Set Forcing is rather special among the possible generalisations of many forcings known that introduce new subsets to $\omega$.

A very surprising result of Vladimir Kanovei in \cite{zbMATH01335192} shows that Perfect Set Forcing can be iterated with countable supports along {\em any} partial order, while preserving $\aleph_1$. This theorem
builds on Kanovei's earlier work in \cite{zbMATH03685477} and is proved by introducing 
a novel technique called {\em geometric iteration}. The technique is very special to Perfect Set Forcing and it uses properties of the descriptive set theory of the reals. The paper \cite{zbMATH03685477} is very interesting from the logical point of view as well, since it gives a novel semantics for Perfect Set Forcing over ${\mathbf L}$, which shows that under the assumption that $(\omega_2)^{\mathbf L}$ is countable in ${\mathbf V}$, the generic extensions simply exists in ${\mathbf V}$. This representation completely dispenses with the need for a countable transitive model of a fragment of ZF in order to do this particular forcing. Kanovei mentions that in this approach he was inspired by the ideas of Ronald Jensen in \cite{zbMATH03387344}.

Our paper extends the iteration results from  \cite{Property-B} to the class of well-founded partial orders
instead of ordinals. For the purposes of this paper we can concentrate on $\FF^\ast_\kappa$, which is the filter of co-bounded subsets of $\kappa$. That is, we can iterate ${\mathbb P} (\mathcal F^\ast_\kappa)$ with supports of size $\le\kappa$ along such partial orders and still preserve cardinals $\le\kappa^+$.
This result can be seen as an analogue of Kanovei's result applicable to the forcing ${\mathbb P} (\mathcal F^\ast_\kappa)$, albeit having to restrict to well-founded partial orders.

Our technique is necessarily rather different from that of \cite{zbMATH01335192}, since the descriptive set theory is no longer available. Instead, we use the ranks of well-founded orders. As to the class of well-founded orders, it includes various trees and more complicated orders. Hence the results here constitute a true generalisation of what was obtained in \cite{Property-B}. The technique also gives a generalisation of the familiar theorem
by James Baumgartner \cite[Th 7.1]{Ba}, who proved that Axiom A forcing iterated with countable supports along any ordinal, preserves $\aleph_1$
\footnote{In fact, some preservation assumptions must be made in this general theorem. For example, if we force with a forcing defined using a parameter, say a filter, we need to assume that the filter remains the same after forcing with Axiom A forcing. We call this property {\em absolutely definable}, see Definition \ref{not:F}.}. We prove in Theorem \ref{cor:aleph_0} that the same is true if such an iteration is made along any well-founded partial order.

In addition to the above results, we discuss the generic objects added by ${\mathbb P} (\mathcal F^\ast_\kappa)$ and show that they include a function in 
${}^\kappa\kappa$ which dominates on a club all such functions from the ground model. Together with the main iteration result, this gives a proof of a result claimed by James Cummings and Saharon Shelah \cite[Th. 1]{zbMATH00819406}. It states that for regular $\kappa>\aleph_0$ it is consistent that the structure ${}^\kappa\kappa$ modulo the order of the domination on a club contains a set isomorphic to any well-founded order given in advance.
This resembles the well known property of dominating reals obtained by Hechler forcing. Our method reproves that result as well.

Finally, we discuss the preservation of $\kappa^{++}$ and show that there are difficulties in that. They stem from the situation which (in spite of much published work on the subject) remains unclear even in much simpler contexts, such as having a direct proof that iterating $\omega_2$ Laver reals with countable supports over a model of CH does not collapse $\aleph_2$.

The main technical difficulty of the paper is that the geometric iteration of Perfect Set Forcing is not really an iteration but a product organised along a possibly non-linear partial order. This works well for the product of Perfect Set Forcing since a product with countable support of such forcings preserves $\aleph_1$. In our case, the analogue is no longer true. In fact, in \cite[Th 5.6]{Property-B} we showed that 
if $2^\kappa=\kappa^+$, then the product of $\kappa$ many copies of ${\mathbb P} (\mathcal F)$ collapses $\kappa^+$, and this is true even for 
$\kappa=\aleph_0$. Our approach therefore was to give up the idea of the products but to use a strong property of well-ordered iterations of 
${\mathbb P} (\mathcal F)$  called {\em amalgamation property}. We proved \cite[Th 7.2]{Property-B} that this property holds for 
the well-ordered iteration of ${\mathbb P} (\mathcal F)$. In this paper we show that it holds for the well-founded iterations too, and we use that to obtain the general iteration theorem.

The paper is organised as follows. \S\ref{sec:definitions} gives basic definitions and facts. \S\ref{sec:difficcult} explains why geometric iterations cannot be used in this case and describes the main difficulties that need to be resolved. \S \ref{sec: ourdef} gives the definitions of the concepts needed from 
\cite{Property-B}. \S\ref{sec:it-wf} develops the definition of an iteration along a well-founded order. \S \ref{sec:main} proves the main result of the paper,
which is that, for appropriate filters $\mathcal F$, we can iterate ${\mathbb P}(\FF)$ along any well-founded order and preserve cardinals up to and including
$\kappa^+$. \S\ref{sec:applications} shows what the generic object of ${\mathbb P}(\FF)$ is. Finally, \S \ref{sec:double+} discusses the difficulties of proving the preservation of
$\kappa^{++}$ even when starting from a model of GCH, where we believe that there are no general positive results to be had.  \S\ref{conclusions} recalls the conclusions of the paper.
 
\section{Basic Definitions}\label{sec:definitions} Throughout, $\kappa$ is an infinite cardinal satisfying $\kappa^{<\kappa}=\kappa$.
We recall the definition of ${\mathbb P} (\mathcal F)$ from \cite[Def 4.1]{Property-B}. As motivation, we mention that it is a generalisation from $\omega$ of the Grigorieff forcing on trees, which itself is a generalisation of Perfect Set Forcing. To generalise Perfect Set Forcing from $\omega$ to uncountable cardinals and obtain an iterable forcing notion, it turns out that one has to pass through the Grigorieff forcing rather than working with Perfect Set Forcing directly. The reasons for this are 
explained in \cite[\S\S4 and 9]{Property-B}.

\begin{definition}\label{gSacks:def}
(1) Let $\FF$ be a $(<\kappa)$-complete filter on $\kappa$. We say that a subtree $p\subseteq {}^{{<\kappa}}\kappa$ is  {\em $\FF$-perfect} if 
\begin{itemize}
\item $\sup\{\dom(s):\,s\in p\}=\kappa$, 
\item for any $s\in p$, there exists
$t\supseteq s$ such that $\{\alpha<\kappa:\, t\frown\alpha\in p\} \in \FF$ (we say that such a $t$  is $\FF$-{\em splitting}) and
\item If $\delta<\kappa$ is a limit ordinal and $s\in {}^{\delta} \kappa$ is such that for all $\alpha<\delta$ we have $s\rest\alpha\in p$, then $s\in p$
(we call this property the {\em closure of} $p$).
\end{itemize}

For $s\in p$ we define its {\em splitting history} as
\[
{\rm deg}_p (s)=\{i < \dom(s):\,(\exists t\in p)\,[t\rest i =s\rest i \mbox{ and }t(i)\neq s(i)]\}.
\]

{\noindent (2)} The $\FF$-perfect tree forcing $\mathbb P(\FF)$ is defined as follows:

The conditions in $\mathbb P$ are $\FF$-perfect subtrees $p$ of ${}^{{<\kappa}}\kappa$, satisfying the following additional conditions:
\begin{description}
\item [{\rm (o)}] For every $s\in p$, either $| \{\alpha<\kappa:\,s\frown \alpha\in p\} | = 1$, or $\{\alpha<\kappa:\,s\frown \alpha\in p\}\in \FF$.
\item [{\rm (i) [Closure of splitting for every node]}] For every $s\in p$, the set ${\rm deg}_p (s)$ is a closed subset of $\dom(s)$.
\end{description}
The order is defined by $p\le q$ iff $p\supseteq q$. We use the convention that $p\le q$ means that $q$ is the stronger condition in terms of forcing.

The {\em stem} of $p$,
denoted $s[p]$, is defined as the unique element $s$ of $p$ such that 
\[
\mbox{no }t\subset s \mbox{ is splitting and }s \mbox{ is splitting}.
\]
(Note that this is well-defined).

\smallskip

{\noindent (3)} For $\zeta<\kappa$ the order $\le_\zeta$ is defined by letting 
$p\le_\zeta q$ if 
\begin{itemize}
\item $p\le q$,
\item for all $s\in p$ with ${\rm otp}({\rm deg}_p(s))< \zeta$ we have $s\in q$.
\end{itemize}

\smallskip

{\noindent (4)} For $p$ as above, we let $[p]$ be the set of $\kappa$-branches of $p$, namely
\[
[p]=\{f\in {}^\kappa\kappa:\,(\forall \alpha<\kappa) f\rest\alpha\in p\}. \eod_{\mbox{\tiny{Def.}} \ref{gSacks:def}}
\]
\end{definition}

Note that $\mathbb P(\FF)$ is indeed a partial order with the least element ${}^{{<\kappa}}\kappa$ and that each $\le_\zeta$ is
a partial order.

\subsection{Geometric Iterations} We recall the relevant definitions from \cite[\S1]{zbMATH01335192} and adapt them to our setting, using the notation of that paper.

\begin{notation}\label{def:letters} Letters $\zeta,\xi$ and $\eta$ stand for partial orders, where an order relation is specified, and $i,j$ stand for individual elements of partial orders. The set $\kpo$ is the set of all (isomorphism types of) partial orders of size $\le\kappa$. 

For a partial order $(\zeta,<)$  and $i\in \zeta$, we define $[<i], [\le i], [>i], [\ge i]$ in the natural way. For example $[<i]=\{j\in \zeta:\,j<i\}$.
\end{notation}

\begin{definition}\label{def:initial-segment} An {\em initial segment} of a partial order $(\zeta,<)$ is a downward-closed subset, that is a set $\xi\subseteq \zeta$ such that for all $j\in \xi$ and $i<j$ we have $i\in \xi$. We assume that the initial segments are ordered by the induced order. We let $\isz$ be the set of all
initial segments of the order $(\zeta,<)$.

We let $\mathcal K$ stand for ${}^\kappa\kappa$ and $\mathcal k$ for ${}^{<\kappa}\kappa$. 
$\eod_{\mbox{\tiny{Def.}} \ref{def:initial-segment}}$
\end{definition}

We recall that the main result of \cite{zbMATH01335192} is the definition of an iteration of Perfect Set Forcing with countable supports along any partial order and the theorem \cite[Th. 1]{zbMATH01335192} stating that such an iteration (or a ``long product'' as it is appropriately called in \cite{zbMATH01335192}) preserves  $\aleph_1$. Since Kanovei's discovery of it, his iteration technique has also been called {\em geometric forcing}. This is the terminology we adopt here. In \S\ref{sec:difficcult} we explain why it is unlikely that a result as general as the results mentioned can be obtained for $\kappa$
uncountable, and why the geometric iteration technique does not straightforwardly generalise to this case.

\section{The difficulties of the generalisation to regular $\kappa>\aleph_0$}\label{sec:difficcult} The main difficulty of the generalisation is the lack of a convenient topological structure. 
We may equip $\mathcal K$ with a topology generated by the initial segments: a basic open set is a set of the form $\{\rho: \, s\subseteq \rho\}$ for $s\in \mathcal k$. In particular, this is a $(<\kappa)$-box topology, rather than the ordinary product topology. The space $\mathcal K$ is neither metrisable nor compact in this topology, or in any other topology that respects the order on $\kappa$ sufficiently to let us identify perfect subsets of $\mathcal K$ as branches of $\mathcal k$.

For $(\zeta, <)\in \kpo$ we may define the topological space $\mathcal K^\zeta$ as follows. Its underlying set is the Cartesian product (we may also work with the lexicographic product, but it turns out to be irrelevant) of copies of $\mathcal K$ along $\zeta$, so the set $\Pi_{i\in \zeta} \mathcal K$ ordered by the ordinary coordinatewise order. The set
${\mathcal K}^\zeta$ can be equipped with various topologies, for example the ordinary Tychonoff product \cite{Tychonoff1930} of the topologies on $\mathcal K$. As it turns out, unlike in the countable case for Perfect Set Forcing, there are no useful topological properties of $\mathcal K^\zeta$ we may exploit.  We explain that now.

An important ingredient of the geometric iteration approach is that every countable power of 
the Cantor space $\mathcal D$ is homeomorphic to $\mathcal D$ itself. The proof of this relies on the compactness of 
$\mathcal D$. In particular, whatever topology we give to $\mathcal K$, we do not have the compactness and do not obtain that 
$\mathcal K$ is homeomorphic to its powers of size $\le\kappa$. We may try to ignore the topological properties of $\mathcal K^\zeta$  and work only with the combinatorial properties. As is known from descriptive set theory generalised to uncountable regular cardinals, for example \cite{FriedmanHyttinenKulikov2014}, not much of the ordinary descriptive set theory carries through. Notably because there is no metric available. What is happening here fits that pattern.

We may optimistically try to continue, saying that the proofs in  \cite{zbMATH01335192} use topology to be able to obtain applications about continuous functions and other subjects coming from descriptive set theory, which is not our concern. So we may try to concentrate only on the properties relevant to the combinatorics of forcing. We immediately arrive at a new difficulty, which is that what this is starting to look like is reducing the arbitrary iterations to the product with $\le\kappa$ supports of ${\mathbb P}(\mathcal F)$. However, this product does not in general preserve $\kappa^+$, even in the case
$\kappa=\aleph_0$, see \cite[Th 5.6]{Property-B}.

The solution comes from the following consideration. While we have shown in  \cite[Th 5.6]{Property-B} that products with $\le\kappa$-supports of ${\mathbb P}(\mathcal F)$ do not preserve $\kappa^+$, we have also shown  \cite[Th 7.2+Cor. 7.3]{Property-B} that the iterations with $\le\kappa$-supports of 
${\mathbb P}(\mathcal F¨^\ast_\kappa)$ (or other carefully chosen filters, as explained in \cite[\S7]{Property-B}) along an ordinal do preserve all cardinals up to and including $\kappa^+$. Therefore we do not aim at an iteration-free solution as in geometric iterations, but we show that the arguments about any iteration of ${\mathbb P}(\mathcal F)$ along well-founded orders, made with supports of size $\le\kappa$, may to a large extent be reduced to the ones used when iterating along an ordinal. This is achieved by using the well-foundedness rank of the order, restricting our attention to appropriately selected filters, as explained in \S\ref{sec:it-wf}.

Another difficulty is to  define what is meant by an iteration of forcing along a partial order $I$ and with supports of size $\le \kappa$. The concept is often credited to Marcia Groszek and Thomas Jech in \cite{zbMATH04187795}. They used it to study iterations of Perfect Set Forcing along any partial order. However, their approach lacks detail and in our opinion is too general to be applicable the way that they intended it. In particular, it is absolutely unclear to us how to handle the dependence of a stage in the iteration on the previous stages, for example if the previous stages form an ill-founded set. Moreover, they themselves in \cite[\S1]{zbMATH04187795}, immediately after giving what is intended as a generalised iteration, pass to the case when all the supports are well-founded. We have developed our own approach, less general than the one of \cite{zbMATH03685477}, since we only do it for well-founded partial orders, but more general as well, since \cite{zbMATH03685477} only deals with Perfect Set Forcing while we deal with several other forcings.

In conclusion, the topological arguments from \cite{zbMATH01335192} that apply to Perfect Set Forcing to have a simple argument for forcing along any partial order, work perfectly well, but only in the context of Perfect Set Forcing. For the forcing of the kind $\mathbb P(\FF)$, we have to go back to the idea of a generalised iteration. Various versions of these were introduced in the literature, including some incorrect ones, starting as early as by Stephen Hechler in \cite{zbMATH03510298}. We introduce the specific definition that we use, adapted to the kind of forcing used here. 

We are only able to develop the theory for iterations of ${\mathbb P}(\FF)$ along well-founded partial orders. Let us also note that this only makes sense under certain absoluteness assumptions on $\FF$, which we explain in \S\ref{sec:it-wf}. If we give up the traditional idea of keeping $\FF$ fixed but allow it, more naturally, to be a $J$-name at every stage $J$ of the forcing, then the technique works just as well, with no restrictions except for the $(<\kappa)$-completeness being forced about each
$\name{\FF}_J$.

\section{Some relevant definitions and facts}\label{sec: ourdef} Throughout, $\chi$ stands for a regular cardinal large enough so that any mathematical object mentioned in the definitions below is an element of $\HH(\chi)$, and $\prec^\ast$ stands for a fixed well-order of $\HH(\chi)$. We recall that we work with an infinite cardinal $\kappa$ such that $\kappa^{<\kappa}=\kappa$. 

\begin{definition}\label{def:B-kappa-properness} We say that a forcing notion $\mathbb P$ is {\em $B$-$\kappa$-proper} if
there exists a sequence $\langle \leq_\zeta:\,\zeta<\kappa\rangle$ of partial orders on $P$, satisfying 
\begin{enumerate}
\item  \(\leq_0 = \leq\), and $\xi<\zeta\implies \le_\xi\supseteq \le_\zeta$, 
    \item  for $\delta$ limit $\in (0,\kappa)$, we have $\le_\delta=\bigcap_{\zeta<\delta} \le_\zeta$,
    \item there is some $y\in \HH(\chi)$ such that 
\begin{itemize}
\item whenever $p\in \mathbb P$, $\zeta^*<\kappa$, $N\prec \HH(\chi)$ with $|N|=\kappa$ and $N^{<\kappa}\subseteq N$, such that $y, p, ( \mathbb P,\langle \le_\zeta:\,\zeta<\kappa\rangle) \in N$,\footnote{Note that these assumptions imply that $\kappa\subseteq N$. This is explained in detail in \cite[Footnote 3]{Property-B}.}
\end{itemize}
there exists $q\ge_{\zeta^*} p$ that is $\mathbb P$-generic over $N$.
$\eod_{\mbox{\tiny{Def.}} \ref{def:B-kappa-properness}}$
\end{enumerate}

\end{definition}

\begin{definition}\label{def:strongclosure} A {\em strongly $(<\kappa)$-closed forcing notion} is a pair $({\mathbb P},F)$, where ${\mathbb P}$
is a forcing notion and
$F$ is a definable function which assigns to any increasing sequence $\bar{p}=\langle p_i:\,i<i^\ast<\kappa\rangle$ of conditions in ${\mathbb P}$ a common upper bound $F(\bar{p})$ of $\bar{p}$.

In this case we say that $F$ {\em witnesses} the strong $(<\lambda)$-closure of ${\mathbb P}$. If no confusion arises, we often just say that
${\mathbb P}$ is strongly $(<\lambda)$-closed when we mean that we have fixed a function $F$ such that $({\mathbb P},F)$ is a strongly $(<\lambda)$-closed forcing notion.
$\eod_{\mbox{\tiny{Def.}} \ref{def:strongclosure}}$
\end{definition}

\begin{definition}\label{def:strongfusion} We say that a forcing notion ${\mathbb P}$ together with a sequence
$\langle \le_\zeta:\,\zeta<\kappa\rangle$ of orders witnessing fusion satisfies {\em strong fusion}, if there exists a definable function $t$ with the following property:

for any sequence $\langle p_\zeta:\,\zeta\in [\zeta^\ast, \delta)\rangle$ for $0<\delta$ limit $\le\kappa$ and some $\zeta^\ast<\delta$ such that 
$p_\zeta \le_\zeta p_{\zeta+1}$ for all $\zeta \in [\zeta^\ast, \delta)$ and $p_\xi=t(\langle p_\zeta:\,\zeta\in [\zeta^\ast, \xi)\rangle)$ for all $\zeta^\ast<\xi$ limit $<\delta$, we have that $t(\langle  p_\zeta:\,\zeta\in [\zeta^\ast, \delta)\rangle)$ is the fusion limit of $\langle p_\zeta:\,\zeta \in [\zeta^\ast, \delta)\rangle$.
$\eod_{\mbox{\tiny{Def.}} \ref{def:strongfusion}}$
\end{definition}

\begin{remark}\label{rem:def} The notion of ``definable'' mentioned in Definitions \ref{def:strongclosure} and \ref{def:strongfusion} is often mentioned in set theory, but rarely explained. We explain exactly the kind of definability needed when one discusses forcing, on the example of {\em absolutely definable} filters in Definition \ref{not:F}.
\end{remark}

\begin{fact}\label{fact:properties-of-individual-forcing} The forcing ${\mathbb P}(\mathcal F)$ is $B$-$\kappa$-proper, strongly $(<\kappa)$-closed, and has
the strong fusion.
\end{fact}

\begin{proof} The strong $(<\kappa)$-closure of ${\mathbb P}(\mathcal F)$ as witnessed by intersections is proved in \cite[Ex. 6.2(1)]{Property-B} and its strong fusion property is proved in  \cite[Ex. 6.4(1)]{Property-B}. In \cite[Cor. 4.13]{Property-B} it is proved that ${\mathbb P}(\mathcal F)$ satisfies property $B(\kappa)$, while in
\cite[Th 6.7]{Property-B} it is proved that every forcing satisfying property $B(\kappa)$ is $B$-$\kappa$-proper.
$\eop_{\ref{fact:properties-of-individual-forcing}}$
\end{proof}

\begin{fact}\label{fact:rank} If $I$ is a well-founded partial order, then there exists a unique ordinal $\alpha^\ast$ such that there is an order-preserving
surjection $\rho:\,I\twoheadrightarrow\alpha^\ast$.
\end{fact}

\begin{proof} If $I=\emptyset$ then we have that $\alpha^\ast=0$, as witnessed by the empty function. Suppose that $I\neq\emptyset$ and define $\rho$ from $I$ into the ordinals by letting
\[
\rho(x)=
\begin{cases}
1 &\mbox{ if }x\mbox{ is minimal in }I,\\
\sup\{\rho(y)+1:\,y<_I x\} &\mbox{ otherwise}.
\end{cases}
\]
Since $I$ is a well-founded set, this definition indeed gives an algorithm for computing the value of $\rho(x)$ for every $x\in I$. Since $I$ is a set, there is a least ordinal $\alpha^\ast$ which is not in the range of $\rho$. We claim that $\alpha^\ast$ is the range of $\rho$. Otherwise, there is
the least $\alpha>\alpha^\ast$ which is in the range of $\rho$. Let $\alpha=\rho(x)$.

By the definition of $\rho$, the ordinal $\alpha$ cannot be limit, since then the interval $[\alpha^\ast, \alpha)\cap \ran(\rho)=\emptyset$, so $\rho(x)$
could not have been defined as $\alpha$. Hence $\alpha=\beta+1$ for some $\beta\ge\alpha^\ast$ such that $\beta\notin \ran(\rho)$. But then again
$\rho(x)$ could not have been defined as $\alpha$. Hence, there is no such $\alpha$, which proves that the range of $\rho$ is $\alpha^\ast$.

Finally, $\rho$ is order-preserving
by the construction, and this property along with the surjectivity of $\rho$ implies that $\alpha^\ast$ is unique.
$\eop_{\ref{fact:rank}}$
\end{proof}

\section{Iterations of ${\mathbb P}(\mathcal F)$ along a well-founded order}\label{sec:it-wf} Our aim is to prove an iteration theorem
for the forcing notion ${\mathbb P}(\mathcal F)$ along any well-founded partial order, for a fixed filter $\FF$. More precisely, for a fixed {\em definition
of a filter $\FF$}. We may say, such as in our main example,``$\FF$ is the filter of the co-bounded subsets of $\kappa$'' and then $\name{\FF}$ names what we obtain by unravelling this definition in the extension.

Therefore, for the notion of iteration of  $\mathbb P(\mathcal F)$ to make sense, we must assume that $\name{\mathcal F}$ remains a $(<\kappa)$-complete filter on $\kappa$ after forcing with any $(<\kappa)$-closed forcing (to be precise, any strongly $(<\kappa)$-closed $B$-$\kappa$-proper forcing). This is guaranteed, in particular, if $\mathcal F$ is a ground-model filter which gains no new elements after such forcing. Since the forcing is $(<\kappa)$-closed,
it adds no sequences of length $<\kappa$ of ground-model objects.
Hence $\name{\mathcal F}_G$
is literally the same filter as it was in the ground model and it remains $(<\kappa)$-complete.

From now on, we let $\FF$ denote such a filter. We summarise our assumptions:

\begin{definition}\label{not:F} (1) Let $\kappa$ be an infinite cardinal satisfying $\kappa=\kappa^{<\kappa}$ and let
$\FF$ be a $(<\kappa)$-complete filter on $\kappa$ such that:
\begin{itemize}
\item there is a formula $D(x; a)$ in the language of set theory with the variable $x$ ranging over $H(\chi)$ and $a\in H(\chi)$ a set parameter, which 
satisfies that $\FF=\{X\subseteq \kappa:\,D(X;a)\}$,
\item moreover, for any strongly $(<\kappa)$-closed $B$-$\kappa$-proper forcing $\mathbb P$, we have
\[
\forces_{\mathbb P}``\{\name{X}\subseteq \check{\kappa}:\,D(\name{X};\check{a})\}=\check{\FF}".
\]
\end{itemize}
{\noindent (2)} Having fixed $D$ as above, we shall simplify the forcing notation and refer to the ${\mathbb P}$-name $\name{D}(\check{\kappa})$ simply by
$\FF$ and we shall say that $\FF$ is {\em absolutely definable}.
 \end{definition}

Note that the second item of Definition \ref{not:F}(1) indeed implies the first, by using $\mathbb P=\{\emptyset\}$.

\begin{definition}\label{def:mixed-iteration} Let $\rho_I:I\to \alpha^\ast$ be the rank function associated to a well-founded order $I$. 
By induction on the ordinal $\alpha^*$ we define what constitutes {\em an iteration along $I$ with supports of size $\le\kappa$} of ${\mathbb P}(\mathcal F)$
and what its limit ${\mathbb P}_{I}$ is. 
We shall denote such an iteration by 
\[
\langle ({\mathbb P}_{J}, \name{{\mathbb P}}(\mathcal F)_J):\, J \mbox{ a strict initial segment of } I\rangle.
\]
In this definition, $J$ ranges over the strict initial segments of $I$ and $\name{{\mathbb P}}(\mathcal F)_J$ stands for a fixed ${\mathbb P}_J$-name
for ${\mathbb P}(\mathcal F)$, which can be defined if we are justified in ${\mathbb P}_{J}$ being a forcing notion.

If $\alpha^*=0$, we have that $I=\emptyset$.
We define ${\mathbb P}_\emptyset=\{\emptyset\}$. Thus, ${\mathbb P}_\emptyset$ is a forcing notion whose minimal (and the only) element is $\emptyset$. We then have 
$\name{{\mathbb P}}(\mathcal F)_\emptyset=
{\mathbb P}(\mathcal F)$, since forcing with ${\mathbb P}_\emptyset$ makes no changes to the ground model.

Suppose that
$\alpha^\ast>0$. Since for every $J$ we have that $\rho\rest {J} $ is a rank function onto an ordinal $<\alpha^\ast$,  by the induction hypothesis we have a definition of what is an iteration along $J$ of ${\mathbb P}(\mathcal F)$ and what is its limit. We assume as an induction hypothesis that:

\begin{itemize}
\item For every $J$, the forcing notion ${\mathbb P}_J$ is the limit with supports of size $\le\kappa$ of an iteration 
$\langle {(\mathbb P}_{K}, \name{{\mathbb P}}(\mathcal F)_K):\,K \mbox{ a proper initial segment of }J \rangle$ of ${\mathbb P}(\mathcal F)$, whose minimal
element is 
\[
\emptyset_{{\mathbb P}_J}= \emptyset \!\frown\! \langle \emptyset_{\name{{\mathbb P}}(\mathcal F)_K}, \,K \mbox{ a proper initial segment of }J\rangle,
\]
\item For every $J$ with a non-empty set $\Gamma$ of maximal elements and any $j\in \Gamma$, 
we have that 
\begin{itemize}
\item ${\mathbb P}_{[\le j]}={\mathbb P}_{[<j]} \ast \name{{\mathbb P}}(\mathcal F)_{[<j]}$, and
\item ${\mathbb P}_J$ is the set of all functions $p$ with domain $J$ such that for some $j\in \Gamma$ we have $p\rest [\le j]\in {\mathbb P}_{[\le j]}$
and $p(i)=\emptyset_{{\name{\mathbb P}(\mathcal F)_{{\mathbb P}_{[<i]}}}}$ for $i\in J \setminus [\le j]$.
(Note that this includes $\emptyset_{{\mathbb P}_J}$.)
\end{itemize}
\item For $J$ with no maximal element, the  elements of ${\mathbb P}_{J}$ are sequences of the form $p=\langle p(j):\,j\in J\rangle$ satisfying the following:\begin{enumerate}
\item for every $K$ a proper initial segment of $J$, the sequence $\langle p(k):\,k\in K\rangle$ is in ${\mathbb P}_K$,
\item each $p(j)$ is a canonical ${\mathbb P}_{[<j]}$-name for a condition in $\name{\mathbb P}_{[<j]}(\mathcal F)$,
\item ${\rm supt}(p)=\{j\in J:\,\neg (p\rest [<j]\forces_{{\mathbb P}_{[<j]}} ``p(j)=\emptyset_{\name{\mathbb P}_{[<j]} (\mathcal F)}")\}$ is a set of size $\le\kappa$.
\end{enumerate}
We have treated the elements of ${\mathbb P}_{J}$ as functions with domain $J$. It is with this in mind that we can define the restrictions of such functions to any subset of $J$.
\item The ordering $\le_J$ on ${\mathbb P}_J$ is defined by $p\le q$ iff the following hold:
\begin{enumerate}
\item if $J$ has a non-empty set $\{j_\gamma:\,\gamma<\gamma^\ast\}$ of maximal elements (necessarily pairwise incomparable), then for some $\gamma<\gamma^\ast$, we have
$p, q\in {\mathbb P}_{[\le j_\gamma]}$ and 
\[
(p\rest[<j_\gamma], p(j_\gamma))\le (q\rest [<j_\gamma], q(j_\gamma))
\]
in the two step iteration ${\mathbb P}_{[<j_\gamma]}\ast \name{\mathbb P}(\mathcal F)_{[<j_\gamma]}$,
\item if $J$ has no maximal element, then for all $j\in J$ we have 
\[
p\rest [<j]\le q\rest  [<j].
\]
\end{enumerate}
 \end{itemize}
Having made all these assumptions, there is only one way of defining ${\mathbb P}_I$ which will preserve them, and it is given by repeating all the items of the 
inductive hypothesis with $I$ in place of $J$. 
Given Definition \ref{not:F}, this definition is well posed iff all the forcings appearing in the inductive hypothesis are strongly $(<\kappa)$-closed $B$-$\kappa$ proper forcings. If this condition is not satisfied, we consider that the definition of the iteration and its limit ${\mathbb P}_I$ is impossible. Otherwise, we say that the iteration and its limit are well-defined.
$\eod_{\mbox{\tiny{Def.}} \ref{def:mixed-iteration}}$
\end{definition}

\begin{remark}\label{rem:why-always-the-same}
We note that the very construction of this definition is such that we are always forcing with the same forcing notion $\name{\mathbb P}(\mathcal F)$.
It is conceivable to replace $\name{\mathbb P}(\mathcal F)$ by some other fixed forcing notion, provided it has the required properties of closure, properness and amalgamation under ordinal-length iteration. In the case of $\kappa=\aleph_0$ we can do so for several examples of Axiom A forcing, some of which appear
in Theorem \ref{cor:aleph_0}. For uncountable $\kappa$, for the moment we do not have examples other than $\name{\mathbb P}(\mathcal F)$. 

Once we give up the idea of forcing always with the same $\name{\mathbb P}(\mathcal F)$ and allow $\mathcal F$ to
change with the iteration, that is to have a name $\name{\mathbb P}(\name{\mathcal F}_J)$, many of our difficulties with definability disappear !  So, in fact we obtain a more general theorem with less work. Wishing to concentrate on a concrete case, we have not formulated such a theorem here, but the version of it for iterations along an ordinal appears as Corollary 7.4. in \cite{Property-B}.

It is also possible that a more elaborate approach might lead to being able to interleave different kinds of forcings into the iteration. 
Since for the moment in the case 
$\kappa>\aleph_0$ we do not even know 
how to do such an iteration along an ordinal, because of the amalgamation property which depends on using the same kind of forcing all along, that situation is beyond the reach of our present methodology.

Apart from the references on Perfect Set Forcing which we already mentioned, in the case of  $\kappa=\aleph_0$ generalised iterations have been studied in the literature on cardinal invariants. We have noticed various inaccuracies in that literature, mostly stemming from repeating without checking the errors we have already mentioned in this paper. While we hope that the specialists in this area will review their relevant papers and correct them, that direction is not our focus here. Our results specifically for case $\kappa=\aleph_0$ are Theorem  \ref{cor:aleph_0} and Observation \ref{obs:Hechler}.
\end{remark}

\begin{notation}\label{not:padding} Suppose that $K$ is an initial segment of $J$ and $p\in {\mathbb P}_{K}$. By the {\em padding of $p$ to 
${\mathbb P}_{J}$} we mean the unique, up to forcing equivalence, condition $q\in {\mathbb P}_{J}$ such that $\supt(q)=\supt(p)$ and for every $j\in \supt(p)$ we have $\forces_{{\mathbb P}_{[<j]}}``q(j)=p(j)"$. We denote $q=p \rest^{-1}_K J$ and $p=q\rest_K J$.
\end{notation}

\begin{observation}\label{lem:projection-forcing} Let $J$ be an initial segment of $I$ and let 
\[
\langle ({\mathbb P}_{K}, \name{{\mathbb P}}(\mathcal F)_K):\, K \mbox{ a strict initial segment of } I\rangle
\]
be an iteration along $I$ with supports of size $\le\kappa$ of ${\mathbb P}(\mathcal F)$.
Then 
\begin{itemize}
\item 
$\langle ({\mathbb P}_{K}, \name{{\mathbb P}}(\mathcal F)_K):\, K \mbox{ a strict initial segment of } J \rangle$
is an iteration along $J$ with supports of size 
$\le\kappa$ of ${\mathbb P}(\mathcal F)$,
\item the padding mapping embeds ${\mathbb P}_{J}$ as a complete
suborder of ${\mathbb P}_{I}$ and the mapping $p\mapsto p\rest_I J$ is a projection.
\end{itemize}.
\end{observation}

\begin{proof}
The proof rests on the fact that if $K$ is an initial segment of $J$, then it is also an initial segment of $I$.
It then suffices to observe that in Definition \ref{def:mixed-iteration} the definitions of 
\[
\langle( {\mathbb P}_{K}, \name{{\mathbb P}}(\mathcal F)_K):\, K \mbox{ a strict initial segment of } J\rangle \mbox{ and }{\mathbb P}_{J} 
\]
are independent of $I$.
$\eop_{\ref{lem:projection-forcing}}$
\end{proof}

Now we adapt to the context used here the definition of iterations with amalgamation that we made for iterations along ordinals in \cite[Def.6.8]{Property-B}.
The partial orders $\le_\zeta$ are as defined by Definition \ref{gSacks:def}(3).

\begin{definition}\label{def:amalgamation} 
A well-defined iteration 
\[
\langle ({\mathbb P}_{J}, \name{{\mathbb P}}(\mathcal F)_J):\, J \mbox{ a strict initial segment of } I \rangle
\]
as in 
Definition \ref{def:mixed-iteration} is said to have {\em amalgamation} if for every relevant $J$, for every $C\in [J]^{<\kappa}$, for every $\zeta<\kappa$, for every dense open set 
$\DD\subseteq {\mathbb P}_J$ and every $p\in {\mathbb P}_J$, there exists $q$ such that:
\begin{itemize}
\item $q\in \DD$,
\item $q\ge p$,
\item for every $j\in J$, we have $q\rest [<j]\forces_{\mathbb P_{[<j]}}``q(j)\ge^{\name{{\mathbb P}}(\mathcal F)_{[<j]}}_{\zeta} p(j)"$.
$\eod_{\mbox{\tiny{Def.}} \ref{def:amalgamation}}$
\end{itemize}
\end{definition}

In \S\ref{sec:main} we prove our main theorem, Theorem \ref{th:iteration}.  
Together with the known facts we shall recall, it implies that iterating
$\mathbb P(\mathcal F)$ along any well-founded partial order preserves cardinals up to and including $\kappa^+$ (see Corollary \ref{cor:preserving-cardinals}).

\section{Iterating $\mathbb P(\mathcal F)$ along well-founded partial orders preserves cardinals $\le\kappa^+$}\label{sec:main}
The following is a general iteration theorem for $\mathbb P(\mathcal F)$. The theorem appearing in \cite[Th 6.12]{Property-B} is a special case of this result when $I$ is an ordinal.

We note, as we did in \cite[Obs. 6.9]{Property-B},  that if the result of an iteration of forcing is  $B$-$\kappa$-proper, then the iteration itself has amalgamation, by definition. The reason for stating Theorem \ref{th:iteration} in this form is that it is proved by induction on the rank of $I$.
At each inductive step, we use the amalgamation property of the iteration obtained so far, together with the new step, along with the inductive hypothesis of $B$-$\kappa$-properness on smaller ranks, to show that the limit of the iteration is $B$-$\kappa$-proper.

\begin{theorem}[General Iteration Theorem]\label{th:iteration} Suppose that $\kappa^{<\kappa}=\kappa$, $I$ is a well-founded partial order and $\mathcal F$ is a $(<\kappa)$-complete filter
on $\kappa$ which is absolutely definable (see Definition \ref{not:F}).

Assume that the iteration $\bar{P}_I=\langle ({\mathbb P}_{J}, \name{{\mathbb P}}(\mathcal F)_J):\, J \mbox{ a strict initial segment of } I \rangle$ made with supports of size $\le\kappa$ is well-defined. Then:
\begin{enumerate}
\item[(i)] the limit ${\mathbb P}_{I}$ of the iteration is a well-defined strongly $(<\kappa)$-closed forcing,
\item[(ii)] the iteration has amalgamation, and
\item[(iii)] ${\mathbb P}_{I}$ is $B$-$\kappa$-proper.
\end{enumerate}
\end{theorem}

\begin{proof} The proof proceeds by induction on the well-foundedness rank $\alpha^\ast$ of $I$. If $\alpha^\ast=0$ the conclusion is vacuously true, so we assume
$\alpha^\ast>0$. We note that for every proper initial segment $J$ of $I$, its well-foundedness rank is $<\alpha^\ast$. 
Hence, we have that for any such $J$, the iteration 
$\langle ({\mathbb P}_{K},  \name{{\mathbb P}}(\mathcal F)_K):\, K \mbox{ an initial segment of } J \rangle $ is a well-defined iteration
along $J$ with supports of size $\le\kappa$. It follows by the induction hypothesis applied on this
iteration that
its limit ${\mathbb P}_{J}$ satisfies the properties stated in the statement of the theorem.

Since we have assumed that $\mathcal F$ is absolutely definable, the above properties of ${\mathbb P}_{J}$ allow us to conclude
that for every relevant $J$, the filter $\mathcal F$ remains a $(<\kappa)$-complete filter on $\kappa$ after the forcing with ${\mathbb P}_{J}$. 
Therefore, by Fact  \ref{def:mixed-iteration}, there is a ${\mathbb P}_{J}$-name
$\name{f}_J$ for a function witnessing that $\name{{\mathbb P}}(\mathcal F)_J$ is strongly $(<\kappa)$-closed. 
\medskip

{\noindent (i)} We prove that ${\mathbb P}_{I}$ is strongly $(<\kappa)$-closed as witnessed by a function $h_{I}$.

Suppose that $\bar{p}=\langle p_\varepsilon:\,\varepsilon<\varepsilon^\ast\rangle$ is an increasing sequence in ${\mathbb P}_{I}$ and
 $\varepsilon^\ast<\kappa$. By induction on the rank
$\alpha<\alpha^\ast$ of a proper initial segment $J$ of $I$, we choose an upper bound 
$q_J=h_J(\langle p_\varepsilon\rest J:\,\varepsilon<\varepsilon^\ast\rangle)$, and we claim that we can do it so that 
for every $K$ an initial segment of $J$, both of which are proper initial segments of $I$, we have $q_K= q_J\rest K$. To be able to do this, note that if $K$ is a proper initial segment of $J$, then its well-foundedness rank is strictly less than that of $J$. The proof proceeds by induction on the well-foundedness rank 
$\alpha$ of $J$. The inductive hypothesis is that for every $K$ a proper initial segment  of $J$, we have chosen
\[
q_K=h_K(\langle p_\varepsilon\rest K:\,\varepsilon<\varepsilon^\ast\rangle)
\]
and hence implicitly, that the functions $h_K$ satisfy that for any $K,K'$ proper initial segments of $J$, with $K$ an initial segment of $K'$, for
every increasing sequence $\bar{r}$ of length $<\kappa$ in ${\mathbb P}_{K'}$, we have that
\[
h_{K'}(\bar{r})\rest K= h_{K}(\overline{r\rest K}).
\]
Notice that this requirement makes sense, since if $\bar{r}$ is increasing in ${\mathbb P}_{K'}$, then $\overline{r\rest K}$ is increasing in ${\mathbb P}_{K}$. We shall in particular define $h_J$ so that this inductive requirement is preserved.

The case \underline{$\alpha=0$} is trivial. Consider the case \underline{$\alpha=\beta+1$}. This means that $J$ has a non-empty set of maximal elements,
$\{j_\gamma:\,\gamma<\gamma^\ast\}$. In particular, these maximal elements are pairwise incomparable in $J$. 
By the definition of the forcing ${\mathbb P}_J$, we have that for two conditions $p,q\in {\mathbb P}_J$ to be
comparable, there must be a $\gamma<\gamma^\ast$ such that $p,q\in {\mathbb P}_{[\le j_\gamma]}$. Since the sequence $\bar{p}$ is increasing, such
a $\gamma$ must be common to all $p_{\varepsilon}$. Let us denote that $j_\gamma$ by $j$. Notice that then we have that for every $\varepsilon$ the condition $p_\varepsilon \rest J$ is in fact $p_\varepsilon \rest {[\le j]}$.

By the induction hypothesis, ${\mathbb P}_{[<j]}$ is a strongly 
$(<\kappa)$-closed forcing, as witnessed by a function $h_{[<j]}$. By the definition of the iteration Definition \ref{def:mixed-iteration}, we have ${\mathbb P}_{[\le j]}={\mathbb P}_{[<j]}\ast \name{{\mathbb P}}(\mathcal F)_{[< j]}$. 
We define
\[
q_J\!=\!h_J(\!\langle p_\varepsilon\rest J:\,\varepsilon<\varepsilon^\ast\rangle\!)\!=\!(h_{[<j]} (\!\langle (p_\varepsilon\rest [<j]:\,\varepsilon<\varepsilon^\ast\rangle\!), 
\name{f}_{[< j]}(\!\langle p_\varepsilon(j):\,\varepsilon<\varepsilon^\ast\rangle\!).
\]
This is a well-defined condition in ${\mathbb P}_{[\le j]}$, hence in ${\mathbb P}_J$, and it is clearly forced by $\mathbb{P}_{[\le j]}$ and so by ${\mathbb P}_J$ to be a common upper
bound to $\bar{p}\rest J$. Moreover, the inductive hypothesis that $q_{[<j]}=h_{[<j]} (\langle (p_\varepsilon\rest [<j]:\,\varepsilon<\varepsilon^\ast\rangle)$
and the inductive hypothesis on its properties, imply that
$q_K= q_J\rest K$ is satisfied for every initial segment $K$ of $J$.

Now suppose that \underline{$\alpha>0$ is a limit ordinal of cofinality $\theta\le\kappa$}. Let $\langle \alpha_\zeta:\,\zeta<\theta\rangle$ be an increasing
sequence of ordinals whose supremum is $\alpha$. Since our definition of $h_J$ needs to be functional,  hence no ambiguity is allowed, we choose 
$\langle \alpha_\zeta:\,\zeta<\theta\rangle$ as the $\prec^\ast_\chi$-first such sequence in $\HH(\chi)$.
For $\zeta<\theta$ define
\[
K_\zeta=\{j\in J:\,\rho_J(j)<\alpha_\zeta\}.
\]
Hence $K_\zeta$ is a well-founded order of rank $\alpha_\zeta$ and it is downward closed in $J$, by the definition of $\rho_J$. Moreover, for $\xi\le\zeta$
we have that $K_\xi$ is an initial segment of $K_\zeta$.

By the inductive assumption we have defined the conditions $\langle q_{K_\zeta}:\,\zeta<\theta\rangle$ using a sequence
$\langle h_{K_\zeta}:\,\zeta<\theta\rangle$ of functions where each  $h_{K_\zeta}$ witnesses the strong $(<\kappa)$-closure of 
${\mathbb P}_{K_\zeta}$, so that $q_{K_\zeta}$ is a common upper bound of $\langle p_{\varepsilon}\rest K_\zeta:\,\varepsilon <\varepsilon^\ast
\rangle$ and so that $\xi<\zeta\implies q_{K_\xi}=q_{K_\zeta}\rest K_\xi$. Now we define $q_J=h_{J}(\bar{p})$ as the unique function $q$ with domain $J$ such that $q_{K_\zeta}=q_{K_\zeta}$ for all $\zeta$. Then $q_J$ is a condition in 
${\mathbb P}_J$, since the fact that $\theta\le\kappa$ implies that the support of $q_J$ has size $\le\kappa$. It follows from the definitions and the inductive assumptions that $q_J$ has all the required properties.

The final case is when \underline{$\cf(\alpha)=\theta>\kappa$}. Let 
$\beta=\sup\{\rho_J (p_\varepsilon\rest J)+1:\,\varepsilon < \varepsilon^\ast\}$, hence $\beta<\alpha$. Let $K=\{j\in J:\,\rho_J(j)<\beta\}$, so $K$ is a proper initial segment of $J$ whose rank is 
$\beta$. By the inductive hypothesis, $h_K$ and $q_K$ have been defined as required. Let 
\[
q_J=h_J(\bar{p}\rest J)=q_K\rest^{-1}_K J.
\]
It follows that $q_J$ is as required.

\medskip

{\noindent (ii)} In order to demonstrate the amalgamation, we  need to define some auxiliary orders and some notation. For $j\in J$ let 
$\langle  \le_\zeta^{{\name{\mathbb P}(\FF)}_{[<j]}}:\,\zeta<\kappa\rangle$ be a fixed sequence of ${\mathbb P}_{[<j]}$-names for the orders witnessing that
$\name{\mathbb P}(\FF)_{[<j]}$ has strong fusion.

\begin{definition}\label{def:F-n} (a) For $J$ an initial segment of $I$ and $C\subseteq J$ of size $<\kappa$, for $\zeta<\kappa$ and $p,q\in {\mathbb P}_J$, we define what it means that 
$p\le_{C,\zeta}^J q$:

$p\le_{C,\zeta}^J q \mbox{ iff }$
\begin{itemize}
\item $p\le q$,
\item $(\forall j \in C)\,q\rest[<j]\forces_{{\mathbb P}_{[<j]}}``p(j)\le_\zeta^{{\name{\mathbb P}(\FF)}_{[<j]}} q(j)"$.
\end{itemize}

We note that $\le_{C,\zeta}^J$ is a partial order with the least element $\emptyset_{{\mathbb P}(\mathcal F)_J}$ and that for $K$ an initial segment of
$J$, an initial segment of $I$, for $p, q\in {\mathbb P}(\mathcal F)_J$, we have $p \le_{C,\zeta}^K q$ iff $p \rest^{-1}_K J \le_{C,\zeta}^J q\rest^{-1}_K J$.

\medskip

{\noindent (b)} A {\em fusion sequence in} ${\mathbb P}_J$ is a sequence of the form 
\[
\langle (p_\zeta, C_\zeta):\,\zeta\le\delta \rangle \mbox{ such that}:
\]
\begin{enumerate}
\item $\delta\le\kappa$ is a limit ordinal,
\item each $p_\zeta\in  {\mathbb P}_J$ and for all $\zeta$ we have $p_\zeta \le_{C_\zeta,\zeta}^J p_{\zeta+1}$,
\item the sequence $\langle C_\zeta:\,\zeta<\delta\rangle$ is an increasing continuous sequence of elements of $[J]^{<\kappa}$,
\item if $j\in  C_{\zeta^\ast}\cap\supt(p_{\zeta^\ast})$, where $\zeta^\ast<\delta$ is the first such ordinal,
then 
\begin{equation}\label{eq:suc-fusion}\tag{*}
p_{\zeta^\ast}\rest [<j]\forces _{\mathbb P_{[<j]}}``\langle p_\zeta(j):\,\zeta\in [\zeta^\ast, \delta)\rangle \mbox{ is a fusion sequence in }
{\mathbb P}(\mathcal F)_{[<j]}",
\end{equation}
\item if $\xi< \delta$ is a limit ordinal and $j \in C_\xi$, then 
\[
p_\xi\rest [<j] \forces_{\mathbb P_{[<j]}}``p_\xi(j)=\name{t}_j (\langle p_\zeta(j):\,\zeta<\xi \rangle)".
\]
\end{enumerate}
\end{definition}

\begin{lemma}[Fusion Lemma]\label{lem:fusion-lem} Let $J$ be an initial segment of $I$, not necessarily proper. Then there is a function $H_J$
acting on fusion sequences in ${\mathbb P}_J$  such that the following holds
for any limit ordinal $\delta\le\kappa$:

if $\langle (p_\zeta, C_\zeta):\,\zeta<\delta\rangle$ is a fusion sequence in
${\mathbb P}_J$, such that for all limit $\xi<\delta$ we have $p_\xi=H_J(\langle (p_\zeta, C_\zeta):\,\zeta<\xi\rangle)$,
then letting $q=H_J(\langle (p_\zeta, C_\zeta):\,\zeta<\delta\rangle)$, we have that
\begin{itemize}
\item $q\in {\mathbb P}_J$ and
\item $p_\zeta\le_{C_\zeta,\zeta}^J q$ holds for any $\zeta<\delta$. 
\end{itemize}
\end{lemma}

\begin{proof}[Proof of Lemma \ref{lem:fusion-lem}] The proof is by induction on $\delta$, for all $J$ simultaneously.
We seek to construct $H_J(\langle (p_\zeta, C_\zeta):\,\zeta<\delta\rangle)$, which we denote by $q^J$.

For $\delta=0$, we let $q^J=\emptyset_{\mathbb P_J}$.

Now suppose that $\delta\in (0,\kappa]$.
We define $q^K\in {\mathbb P}_K$ for $K$ an initial segment of $J$ by induction on the well-foundedness rank $\alpha$ of $K$. The requirements of the induction are:
\begin{itemize}
\item for any $\zeta<\kappa$ we have $q^K\ge_{C_\zeta\cap K, \zeta}^K \,p_\zeta\rest K$,
\item for $K$ an initial segment of $K'$ an initial segment of $J$,  we have $q^{K'}\rest K =q^K$.
\end{itemize}

The induction steps are similar to those in the proof of (i), as follows.

\underline{$\alpha=0$.} We let $q^0=\emptyset_{{\mathbb P}_\emptyset}$.

\underline{$\alpha=\beta+1$}. This means that $K$ has a non-empty set of maximal elements,
$\{j_\gamma:\,\gamma<\gamma^\ast\}$. These maximal elements are incomparable in $K$. 
We have that for two conditions $p,q\in {\mathbb P}_K$ to be
comparable, there must be a $\gamma<\gamma^\ast$ such that $p,q\in {\mathbb P}_{[\le j_\gamma]}$.

Since the sequence $\langle p_\zeta\rest K:\,\zeta<\delta\rangle$ is increasing in ${\mathbb P}_K$, such
a $\gamma$ must be common to all $p_\zeta\rest K$. Let us denote that $j_\gamma$ by $j$. Notice that then we have that for every $\zeta$ the condition $p_\zeta \rest K$ is in fact $p_\zeta \rest {[\le j]}$.
We have assumed that $q^{[<j]}$ has been defined as required. 

If $j\notin \bigcup_{\zeta<\delta} \supt (p_\zeta)$, we let $q^K=q^{[<j]} \rest_{[<j]}^{-1} K$.
The conclusion then follows by the
induction hypothesis.

Otherwise, $j \in  \bigcup_{\zeta<\delta} \supt (p_\zeta)$, hence we can find a minimal $\zeta<\delta$
such that $j \in \supt(p_\zeta)$, call it $\zeta^\ast$. In ${\mathbf V}^{{\mathbb P_{[<j]}}}$ we shall define $q(j)\in \mathcal{P}(\FF)$ using the function $t$
from Definition \ref{def:strongfusion} on a sequence of the form
$\bar{r}=\langle r_\zeta:\,\zeta<\delta\rangle$, which we define as follows. For $\zeta\le\zeta^\ast$ we define $r_\zeta=p_{\zeta^\ast}(j)$ and for 
$\zeta>\zeta^\ast$ we define $r_\zeta=p_\zeta(j)$. By equation (\ref{eq:suc-fusion}) in the definition of a fusion sequence, it follows that we have defined a fusion sequence.

We therefore define $q^{[\le j]}=(q^{[<j]}, \name{t}(\bar{r}))$. Then let $q^K=q^{[\le j]} \rest_{[\le j]}^{-1} K$.

\underline{$\alpha$ is a limit ordinal of cofinality $\le\kappa$.} We construct $q^K$ in the exactly same way as in the proof of (i) of Theorem \ref{th:iteration}, where we replace $J$ by $K$ and $h_{K_\zeta}$ by $H_{K_\zeta}$. Since we have already used the index $\zeta$ for the fusion sequence, we shall
write $\langle\alpha_\varepsilon:\,\varepsilon <\cf(\alpha)\rangle$ for the chosen cofinal sequence of $\alpha$ used in the construction. At the end, $q^K$ is the unique condition in ${\mathbb P}_K$ satisfying that  for any initial segment $K'$ of $K$ we have $q^K \rest {K'}=q^{K'}$. Hence, as in (1), $q^K$ is a well-defined condition in ${\mathbb P}_K$. We have to check that
$p\le^K_{C_\zeta\cap K, \zeta} q^K$ holds for each $\zeta$. 

Note that $K=\bigcup_{\varepsilon <\cf(\alpha)} K_{\alpha_\varepsilon}$. Clearly, we have $p\le q$, since for every
$\varepsilon$ we have $p\rest K_{\alpha_\varepsilon}\le q\rest  K_{\alpha_\varepsilon}$. Let $\zeta<\delta$ be arbitrary and suppose $j\in K\cap C_\zeta$. Hence there is  $\varepsilon <\cf(\alpha)$ such that $j\in K_{\alpha_\varepsilon}\cap C_\zeta$. By the induction hypothesis we have 
$q\rest [<j]\forces_{{\mathbb P}_[<j]}
``p(j)\le_\zeta^{{\name{\mathbb P}(\FF)}_{[<j]}} q(j)"$, as required.

\underline{$\beta$ is a limit ordinal of cofinality $>\kappa$.} This case is handled trivially by the requirement on the size of the supports. By that 
requirement we have that there exist an initial segment $K'$ of $K$ such that $\bigcup_{\zeta<\delta}\supt(p_\zeta)\subseteq K'$. 
Hence we have that for every $\zeta<\delta$ the equality $p_\zeta=p_\zeta\rest^{-1}_{K'} K$ holds and so it suffices to let  $q^K=q^{K'}\rest^{-1}_{K'} K$.
$\eop_{\ref{lem:fusion-lem}}$
\end{proof}

{\noindent (iii)} For this proof, we assume that we have proven (i)-(ii) for ${\mathbb P}_I$ and we prove (iii). The idea is to
to use the assumption on the iteration having amalgamation to boost the proof that its limit is $B$-$\kappa$-proper. The proof is essentially identical to that of \cite[Lem. 6.14]{Property-B}; we isolate the key lemma for future reference. In Lemma
\ref{lem:amalg-implies-proper} we allow the possibility that $\name{Q}_J$ are not necessarily of the form $\name{\mathbb P}(\FF)_J$, but any 
individual forcings that would give us the required properties of the iteration. What is meant by an iteration is analogous to Definition \ref{def:mixed-iteration}.

\begin{lemma}\label{lem:amalg-implies-proper} Suppose that $I$ is a well-founded order and that the iteration
\[
\langle ( {\mathbb P}_J, \name{Q}_J):\,J  \mbox{a proper initial segment of }I\rangle
\]
made with supports of 
size $\le\kappa$ is well-defined, strongly $(<\kappa)$-closed, has strong fusion and amalgamation. Then its limit  ${\mathbb P}_I$ is 
$B$-$\kappa$-proper.
\end{lemma}

\begin{proof}
Let 
\[
 y=\langle\bar{P}_I=\langle ({\mathbb P}_J, \name{Q}_J):\,J \mbox{ a proper initial segment of }I\rangle, \kappa , I\rangle
 \]
 and let $H$ be the $\prec^\ast_\chi$-first function witnessing the
amalgamation of $\bar{P}_I$. Suppose that
$N\prec (\HH(\chi), \prec^\ast)$ with $|N|=\kappa$ and $N^{<\kappa}\subseteq N$ is such that $y, p \in N$. It follows that
${\mathbb P}_I, H\in N$. 

Let $\langle \DD_\zeta:\;\zeta<\kappa\rangle$ be an enumeration of all dense open sets of ${\mathbb P}_I$ which are in $N$. Let $N\cap I=
\bigcup_{\zeta<\kappa} C_\zeta$ be a continuous non-decreasing union where for each $\zeta$ we have $|C_\zeta|<\kappa$. In particular, each $C_\zeta\in N$.

By induction on $\zeta<\kappa$ we construct 
a fusion sequence
$\langle (p_\zeta, C_{f(\zeta)}):\,\zeta<\kappa\rangle$ such that
\begin{itemize}
\item $p_0=p$, $p_\zeta\in N$, 
\item $p_{\zeta+1}\in \DD_\zeta$,
\item for any limit $0<\delta<\kappa$ we have $p_\delta=H(\langle (p_\zeta, C_{f(\zeta)}):\,\zeta<\delta\le\kappa\rangle)$.
\end{itemize}
This induction is made possible by the fact that if $p_\zeta\in N$, then $\supt(p_\zeta) \in N$. By elementarity we have that there is a surjection $g:\,\kappa
\twoheadrightarrow\supt(p_\zeta)$ with $g\in N$, so for every $\alpha<\kappa$ we have $g(\alpha)\in N$. We conclude that $\supt(p_\zeta)\subseteq N\cap I$.

To achieve the successor stages of the induction, we shall use the property of amalgamation, applied within $N$. That property guarantees that 
we can find  $p_{\zeta+1}\in N\cap \DD_\zeta$ with $p_\zeta \le_{C_\zeta, \zeta} p_{\zeta+1}$.

We let $f(\zeta+1)$ be the least $\xi$ such that $\supt(p_{\zeta+1})\subseteq C_\xi$.

At non-zero limit stages $\delta$ we
observe that the construction is done so that 
$\langle (p_\zeta, C_{\zeta}):\,\zeta<\delta\rangle\in N$ and the function $H$ applies to the sequence $\langle (p_\zeta, C_{\zeta}):\,\zeta<\delta\rangle$, so we let $p_\delta=H(\langle (p_\zeta, C_\zeta):\,\zeta<\delta\rangle)$. By elementarity, $p_\delta$ is an element of $N$.

At the end, we let $q=H(\langle (p_\zeta, C_{f(\zeta)}):\,\zeta<\kappa\rangle)$.
We shall show that $q$ is ${\mathbb P}_I$-generic over $N$. 

Let $\DD$ be a dense open subset of ${\mathbb P_{I}}$ with $\DD\in N$. We claim that $\DD\cap N$ is predense above $q$. So let $r\ge q$
and let $\DD=\DD_\zeta$ for some $\zeta$. We need to find a condition in $\DD\cap N$ comparable with $r$ and we claim that $p_{\zeta+1}$
is such a condition. We have that $p_{\zeta+1}\le_{C_{\zeta+1}, {\zeta+1}} q$ and so clearly $p_{\zeta+1}\le r$. 
$\eop_{\ref{lem:amalg-implies-proper}}$
\end{proof}

$\eop_{\ref{th:iteration}}$
\end{proof}

\begin{corollary}\label{cor:preserving-cardinals} Let $\kappa$ and $\mathcal F$ be as above. The iteration of ${\mathbb P}(\mathcal F)$ along any well-founded partial order preserves all cardinals up to and including $\kappa^+$.
\end{corollary}

\begin{definition}\label{def:axiomA} In analogy with Definition \ref{not:F} with $\kappa=\aleph_0$, replacing $B$-$\kappa$-proper by Axiom A, we define what is an absolutely definable Axiom A forcing. For such a forcing, we define an iteration of ${\mathbb P}$ along
a well-founded partial order $I$ analogously to Definition \ref{def:mixed-iteration}, but replacing ${\mathbb P}(\mathcal F)$ everywhere by ${\mathbb P}$.
$\eod_{\mbox{\tiny{Def.}}\ref{def:axiomA}}$
\end{definition}

\begin{theorem}\label{cor:aleph_0} Let $\mathbb P$ be an absolutely definable Axiom A forcing with strong fusion and $I$ a well-founded partial order. Then 
the iteration of ${\mathbb P}$ along $I$ preserves $\aleph_1$.
\end{theorem}

\begin{proof} This proof follows the lines of the proof of Theorem \ref{th:iteration}, but is simpler. First we notice that for $\kappa=\aleph_0$, being 
$B$-$\kappa$-proper simply means satisfying Axiom A and that being strongly $(<\kappa)$-closed is trivially true for any forcing. Coming to part (ii) of the proof of Theorem \ref{th:iteration} and Definition \ref{def:F-n}, we make an analogous definition of fusion sequences for ${\mathbb P}$, simply using the orders 
$\langle \le_\zeta:\,\zeta<\omega\rangle$ that demonstrate the Axiom A property of $\mathbb P$. Lemma \ref{lem:fusion-lem} has literally the same proof, only replacing $\mathbb P$ for ${\mathbb P}(\mathcal F)$ everywhere. The same is true for Lemma \ref{lem:amalg-implies-proper}.
$\eop_{\ref{cor:aleph_0}}$
\end{proof}

We note that the common notions of arboreal forcing (Laver forcing, Miller forcing and so on) all satisfy the requirements of Theorem \ref{cor:aleph_0}.

\section{The Generic Function and Club Domination}\label{sec:applications} We discuss the generic objects added by ${\mathbb P}(\mathcal F^\ast_\kappa)$ and give an application of the 
General Iteration Theorem. For $\kappa=\aleph_0$ many applications of various forcing notions to change the structure of ${}^{<\omega}\omega$ are known. In particular, Perfect Set Forcing satisfies properties stronger than the Generalised Perfect Set Forcing and we can refer the reader to \cite{zbMATH01335192} for the relevant theorems. There appears to be little existing work on iterations of the tree version of Grigorieff forcing. Thus, applying our results to the case $\kappa=\aleph_0$ may yield new results concerning cardinal invariants of the continuum. Apart from Observation \ref{obs:Hechler},
in this paper we concentrate on the case
$\kappa>\aleph_0$, with the requirement $\kappa^{<\kappa}=\kappa$ as before. We shall also fix $\mathcal F$ to be the filter $\FF^\ast_\kappa$ of co-bounded subsets of
$\kappa$ since this case is representative. Note that this filter is absolutely definable (see Definition \ref{not:F}) by the formula 
\[
\FF^\ast_\kappa=\{X\subseteq \kappa:\,|\kappa\setminus X|<\kappa\}.
\]
Recall the definition of the stem $s[p]$ of a condition $p\in {\mathbb P}({\mathcal F}^\ast_\kappa)$ from \S\ref{sec:definitions}.  
It is proved in \cite[Lem 4.4]{Property-B} that if $G$ is a generic filter for ${\mathbb P}({\mathcal F}^\ast_\kappa)$, then
\[
g=\bigcup\{s[p]:p \in G\}
\]
is a function from $\kappa\to\kappa$. This function is
{\em the generic function added by}  ${\mathbb P}({\mathcal F}^\ast_\kappa)$. 

\begin{theorem}\label{what-added} The generic function $g$ added by  ${\mathbb P}({\mathcal F}^\ast_\kappa)$  satisfies that
for every function $f\in
{}^{\kappa}\kappa\cap {\mathbf V}$,  there is a club $C$ of $\kappa$ in ${\mathbf V}[G]$ 
such that for every $\alpha\in C$ we have $f(\alpha) <g(\alpha)$.
\end{theorem}

\begin{proof} Recall the notation for ${\rm deg}_p (s)$ from \S\ref{sec:definitions}.
We first show that the set 
\[
\DD_f=\{p\in {\mathbb P}({\mathcal F}^\ast_\kappa):\,(\forall s\in p)(\forall i\in {\rm deg}_{p}(s))(s(i)>f(i))\}
\]
is dense in ${\mathbb P}({\mathcal F}^\ast_\kappa)$. Given any $p\in {\mathbb P}({\mathcal F}^\ast_\kappa)$, we attempt to construct $q\ge p$ with $q\in \DD_f$ as follows.
By induction on the height $\alpha$ of $s^\ast\in p$ we decide if $s^\ast\in q$. The induction hypothesis carried through the construction is that:
\begin{itemize}
\item  for every $t\in q$
of height $\beta<\alpha$, there is $s\in q$ of height $\beta+1$ with $s\rest\beta=t$,
\item for every $s\in p$ with $\dom(s)$ a limit ordinal $\delta<\alpha$, if for all $\beta<\delta$ we have $s\rest\beta\in q$, then $s\in q$,
\item for every $s\in q$, 
\[
{\rm deg}_q (s)={\rm deg}_p (s).
\]
\end{itemize}
If $\alpha$ is a limit ordinal, $s^\ast\in q$ iff $s^\ast\rest\beta\in q$ for all $\beta<\alpha$. Since $p\in {\mathbb P}({\mathcal F}^\ast_\kappa)$, we have that for any
$s\in {}^{<\kappa}\kappa$ such that $s\rest\beta\in p$ for all $\beta<\alpha$, also $s\in p$. Hence the second item of the induction hypothesis on 
$q$ is preserved at this step.
If $\alpha=\beta+1$, if $t=s^\ast\rest\beta\in q$ and 
$s^\ast\rest\beta$ is not a splitting node of $p$, then let $s^\ast\in q$. Otherwise, the set $\{i<\kappa:\,t\!\frown\! i\in p\}\in \FF^\ast_\kappa$. For $s\in p$ of the form
$t\!\frown\! i$ for some $i$, if  $i>f(\beta)$ we let $s\in q$, otherwise $s\notin q$. Notice that the splittings of $t$ form a co-bounded subset of $\kappa$. At any rate, the first item of the induction hypothesis is preserved by this step. 
Also note that the construction is done so that in every step we have preserved the third item of the inductive hypothesis.

We now verify that $q\in {\mathbb P}({\mathcal F}^\ast_\kappa)$.
Firstly, by induction on $\alpha<\kappa$ we verify that there exist $s\in q$ with $\dom(s)\ge\alpha$. If $\alpha=0$, this is trivial since $\langle\rangle\in q$.
If $\alpha=\beta+1$ let $t\in q$ with $\dom(t)=\beta$, which exists by the induction hypothesis.  By the defining property of $q$, we have that there is
$s\in q$ with $s\rest\beta=t$. Suppose then that $\alpha>0$ is a limit ordinal. By induction on $\beta<\alpha$ we choose an increasing sequence 
$\langle s_\beta:\,\beta<\alpha\rangle$ of elements in $q$ with $\dom(s_\beta)=\beta$. At $\beta$ a successor ordinal we do as in the above
case $\alpha=\beta+1$. For $\beta$ limit we choose $s_\beta=\bigcup_{\gamma<\beta} s_\gamma$, which is in $q$ by the second item in the inductive hypothesis of the construction of $q$. Once this construction is completed, it suffices to let $s=\bigcup_{\beta<\alpha} s_\alpha$.

Now we check that every $s\in q$ has a splitting extension in $q$. Let $t\supseteq s$ be the element of $p$ of the least height which is splitting in $p$.
Our construction at successor cases was done so that $t\in q$ and that $t$ is splitting in $q$, hence the requirement is satisfied. We similarly show that the set of splitting nodes is closed along every branch of $q$.

Note also that the construction of $q$ is such that every splitting is done into co-boundedly many extensions. Hence we conclude that $q\in {\mathbb P}({\mathcal F}^\ast_\kappa)$. The definition of $q$ is done so that $q\in \DD_f$, hence this finishes the proof that $\DD_f$ is dense. Notice that $\DD_f$ is open.

Now let $G$ be a ${\mathbb P}({\mathcal F}^\ast_\kappa)$-generic filter. We work in ${\mathbf V}[G]$ and show that 
\[
\{i<\kappa:\,f(i)<g(i)\}
\]
contains a club of $\kappa$. 

Let $q\in G\cap \DD_f$. First, we show that
$g$ is a branch of $q$. Fix an arbitrary $i<\kappa$ and 
let $r_i\in G$ be  such that $\dom(s[r_i])> i$. Since both $r_i$ and $q$ are in $G$, they are compatible by a condition in $G$, which we shall call $r$. It follows that $s[r]\in q$ and that $g \rest i= s[r]\rest i$. Therefore $g\rest i\in q$. Since $i$ was arbitrary, we obtain that $g\rest i\in q$ for all $i<\kappa$ and, hence, $g$ is a branch of $q$. By induction on $\zeta<\kappa$ we define an increasing sequence $\langle i_\zeta:\,\zeta<\kappa
\rangle$ of ordinals $<\kappa$ as follows.

Let $i_0=\dom(s[q])$ and $s_0=s[q]$. Given $i_\zeta$, $s_\zeta$, let $i_{\zeta+1}$ be the first $i>i_\zeta$ such that there exists $t\supseteq s_\zeta$ in $q$ which splits at level $i$. Let $s_{\zeta+1}$ be such a $t\frown j$, for some $j$ such that $t\frown j \in q$.
For $\zeta$ a limit ordinal $>0$, let $i_\zeta=\sup_{\xi<\zeta} i_\xi$. In particular, $C=\{i_\zeta:\,\zeta<\kappa\}$ is a club of $\kappa$. We claim that $g(i)>f(i)$
for every $i\in C$. We check this by induction on $\zeta$ such that $i=i_\zeta$.

For $\zeta=0$, we have that $g\rest i_0$ is the stem of $q$ and hence $\{j:\,g\rest i_0\!\frown \! \langle (i_0, j)\rangle  \in q \} \in \FF^\ast_\kappa$ and by the choice of $q$ we have that 
every element of $\{j:\,g\rest i_0\!\frown\! \langle (i_0, j)\rangle \in q\}$ is $>f(i_0)$. Since $g\rest (i_0+1)$ is of the form $g\rest i_0\!\frown \! \langle (i_0, j)\rangle $ for some $j$ such that $g\rest i_0\!\frown\! \langle (i_0, j)\rangle \in q$, we conclude that $g(i_0)>f(i_0)$. A similar argument applies to ordinals
of the form $i_{\zeta+1}$. If $\zeta$ is a limit ordinal then we have that $g\rest({i_\zeta+1})\in q$ and $\{i_{\xi+1}:\,\xi<\zeta\}\subseteq i_\zeta$ are
splitting nodes of $g\rest({i_\zeta+1})$. By the requirement of the closure of splitting of every node, (b) in Definition \ref{gSacks:def}, we have that $i_\zeta$ is a splitting node of $g\rest (i_\zeta+1)$. Then we argue as in the successor case to conclude that $g(i_\zeta)>f(i_\zeta)$.
$\eop_{\ref{what-added}}$
\end{proof}

Theorem \ref{what-added} has applications regarding the possible structure of the subsets of $\kappa$ in the generic universe. Here is an example.

\begin{notation}\label{nota:club} For functions $f, g\in {}^{\kappa}\kappa$ we write $f<_{\rm{cl}}^\kappa g$ if $\{i<\kappa:\,f(i)< g(i)\}$ contains a club of $\kappa$.
\end{notation}

The relation $<_{\rm{cl}}^\kappa$ was first considered by Cummings and Shelah in \cite{zbMATH00819406}. They considered the bounding and the dominating number of the structure $({}^{\kappa}\kappa, <_{\rm{cl}}^\kappa)$. The latter one, denoted ${\mathfrak d}_{\rm{cl}}^\kappa$, is the one
relevant here. The following Theorem \ref{th:On-d} is an application of our main theorem Theorem \ref{th:iteration}, and it shows that the order $({}^{\kappa}\kappa, <_{\rm{cl}}^\kappa)$ may consistently be made to embed any well-founded partial order. This theorem was claimed by Cummings and Shelah \cite[Th. 1]{zbMATH00819406}. They used a generalisation of Hechler forcing to obtain the result, but unfortunately their technique of iteration along well-founded orders has a problem, by an argument similar to that of Observation \ref{lem:antichains-of-I}.

\begin{theorem}\label{th:On-d} Suppose that $\kappa=\kappa^{<\kappa}>\aleph_0$ and that $I$ is a well-founded partial order. Then there is a generic
extension preserving cardinals $\le\kappa^+$ in which there is an order preserving embedding $e:\,I\to ({}^{\kappa}\kappa, <_{\rm{cl}}^\kappa)$.
\end{theorem}

\begin{proof} Under the assumptions, consider the generic extension ${\mathbf V}[G]$ obtained by forcing by the limit ${\mathbb P}_I$ of the iteration
with supports of size $\le\kappa$ given by
\[
\langle ({\mathbb P}_J,  \name{{\mathbb P}}({\mathcal F}^\ast)_J):\,J \mbox{ a proper initial segment of }I\rangle.
\]
For each $j\in I$ let $g_j$ be the generic
function added by $\name{{\mathbb P}}({\mathcal F}^\ast)_{<j}$ when forcing over ${\mathbf V}^{\mathbb P}_{[<j]}$. Note that if $i<j$ holds in 
$I$, then by the definition of the iteration we have that $g_i\in {\mathbf V}^{\mathbb P}_{[<j]}$.
Theorem \ref{what-added} shows that then we have that $g_i <_{\rm{cl}}^\kappa g_j$ holds in ${\mathbf V}^{\mathbb P}_{[\le j]}$.
Since the relation $g_i <_{\rm{cl}}^\kappa g_j$ is upward absolute to forcing, we have that $g_i <_{\rm{cl}}^\kappa g_j$ holds in ${\mathbf V}[G]$
as well.
$\eop_{\ref{th:On-d}}$
\end{proof}

\begin{observation}\label{obs:Hechler}
Applying the technique of the proof of theorem \ref{th:On-d} with $\kappa=\aleph_0$ gives a proof of a theorem somewhat generously attributed to Stephen Hechler
\cite{zbMATH03510298}, namely that it is consistent for the structure of $({}^\omega \omega, <^\ast)$ to isomorphically embed any well-founded partial order given in advance.
\end{observation}

In view of further possible applications, in \S\ref{sec:double+} we discuss the preservation of cardinals $>\kappa+$. This turns out to be rather unlikely to happen if the iteration is of large enough well-foundedness rank.

\section{The chain condition of the partial orders ${\mathbb P}_I$ and the problematic history of preservation theorems}\label{sec:double+}
We discuss the possible preservation of cardinals $>\kappa^+$ by orders of the type ${\mathbb P}_I$, as studied above. To set a context in which a preservation result
might be possible, we assume that GCH holds in ${\mathbf V}$. 
In \cite[Lem. 4.15(3)]{Property-B}, it is shown that under these assumptions
${\mathbb P}(\mathcal F)$ has size $\kappa^+$ and, in particular, that it satisfies $\kappa^{++}$-cc and preserves all cardinals $\ge\kappa^{++}$. 
We first note, in Lemma \ref{lem:not-too-long} 
and Corollary \ref{cor:why-collapse} below, that there are some necessary conditions on $I$ even to attempt a preservation result via the iteration.

\begin{lemma}\label{lem:not-too-long}  (1) There is an antichain of size $2^{\aleph_0}$ in the Grigorieff forcings
${\mathbb P}({\mathcal F})$, for any filter $\mathcal F$ over $\omega$.

{\noindent (2)} There is an antichain of size $2^{\aleph_0}$ in Miller forcing, which is the forcing with subtrees of ${}^{<\omega}{\omega}$ where every node has an extension that splits into infinitely many points.

{\noindent (3)} (see \cite[pg. 12]{GeshckeQuickert}) There is an antichain of size $2^{\aleph_0}$ in Perfect Set Forcing.

{\noindent (4)} Suppose that $\kappa^{<\kappa}=\kappa>\aleph_0$ and $\FF$ is a $(<\kappa)$-complete filter over $\kappa$ which contains at least one set
which is not $\kappa$. For example $\kappa\setminus \{0\}\in \FF$. Then there is an antichain of size $2^{\kappa}$ in ${\mathbb P}({\mathcal F})$.

\end{lemma}

\begin{proof} (1) By \cite[Th. II1.2]{Kunen}, there exists a sequence $\langle A_i:\,i<2^{\aleph_0}\rangle$ of elements of $[\aleph_0]^{\aleph_0}$ such that whenever $i<j$ then $|A_i\cap A_j|<\aleph_0$. For each $i$ let $p_i$ be the perfect subtree of ${}^{<\omega}\omega$ defined in the following way.

By induction on $n<\omega$ we define $p_i\cap {}^{n}\omega$. For $n=0$, we have $\langle \rangle\in p_i$. For $n+1$, if $n\in A_i$, then for every $s\in 
{}^{n}\omega\cap p_i$ and $k<\omega$ we let $s\frown k\in p_i $. Otherwise, we let $p_i\cap {}^{n+1}\omega=\{s\frown 0:\,s\in {}^{n}\omega\cap p_i\}$.
Then each $p_i$ is ${\mathcal F}$-perfect, because every splitting is in the set into a filter and every node has a splitting extension since 
$A_i$ is infinite.

However, if $i\neq j$ then there is $n<\omega$ such that $A_i\cap A_j\subseteq n$. We claim that $p_i\cap p_j$ does not have any splitting nodes of
height $>n$. Suppose for a contradiction that $s$ were such, with $\dom(s)=m>n$. Since $s$ is splitting in $p_i$, we must have $m\in A_i$, and similarly
$m\in A_j$. So $m\in A_i\cap A_j$, a contradiction.

{\noindent (2)} The same proof.

{\noindent (3)} A similar proof, except that rather than splitting into an infinite set at the relevant nodes, we split into two values.

{\noindent (4)}  \cite[Th. II1.2]{Kunen} applies to any $\kappa$ satisfying $\kappa^{<\kappa}=\kappa$. Hence, there exists a sequence $\langle A_i:\,i<2^{\kappa}\rangle$ of elements of $[\kappa]^{\kappa}$ such that whenever $i<j$ then $|A_i\cap A_j|<\kappa$. However,
the proof of (1) needs to be modified in this case, because copying the proof literally would give a problem with the fact that the levels of splittings on every branch form a club. Hence, their pairwise intersections are non-empty, unlike in the cases (1)-(3).

For $i<2^\kappa$ let ${\rm cl}(A_i)$ be the closure of $A_i$, that is 
\[
{\rm cl}(A_i)=\{\alpha\mbox{ limit }<\kappa:\;\sup(A_i\cap\alpha)=\alpha\}\cup A_i.
\]

We build the conditions $p_i$ as in (1), by constructing $p_i\cap {}^\alpha\kappa$ by induction on $\alpha<\kappa$. For $\alpha=0$ we take $\langle \rangle$ and
for $\alpha>0$ limit we take 
\[
\{s\in  {}^\alpha\kappa:\,(\forall \beta<\alpha) s\rest\beta\in p_i\}.
\]
Suppose that $\alpha=\beta+1$. If $\beta\in {\rm cl}(A_i)$, for every $t\in {}^\beta\kappa\cap p_i$ we let $t\frown \zeta\in p_i$ for all $\zeta\in \kappa\setminus
\{0\}$. Otherwise, we only let $t\frown 0\in p_i$ for $t\in {}^\beta\kappa\cap p_i$.

It follows that $p_i$ is a condition in ${\mathbb P}(\FF)$. Namely, clearly $\emptyset\neq p_i\subseteq {}^{<\kappa}\kappa$, we have made sure that there are no terminal nodes in $p_i$, so its height is $\kappa$. We also assured at the limit levels that we get the closure. All the splittings are in $\FF$, since $\kappa\setminus \{0\} \in \FF$. Finally, if $s\in p_i$, let $\alpha=\min (A_i\setminus \dom(s))$. Then there is $t\supseteq s$ on level $\alpha$, and it is splitting by the definition. Since we have assured that all elements at any level $\delta\in  {\rm cl}(A_i)$ are splitting, we also obtain the closure of splitting.

We claim that if $i\neq j$ then $p_i$ and $p_j$ are incompatible. Suppose otherwise, so $p_i\cap p_j$ is an $\FF$-perfect tree. Let $s\in p_i\cap p_j$ be
such that $\dom(s)\nsubseteq A_i\cap A_j$, which exists since $\sup(A_i\cap A_j)<\kappa$. By extending $s$ if necessary, without loss of generality $\dom(s)$ is a successor ordinal $\alpha+1$ and $\alpha\in \dom(s)\setminus (A_i\cap A_j)$. Suppose $\alpha\in A_i\setminus A_j$, as the other case is symmetric. Then $s\rest\alpha$ is a splitting node in $p_i$ and $s(\alpha)\neq 0$.
However, $s\rest\alpha$ is not a splitting node in $p_j$ and so $s(\alpha)=0$. A contradiction.
$\eop_{\ref{lem:not-too-long}}$
\end{proof}

\begin{corollary}\label{cor:why-collapse} Let $\FF$ be any absolutely definable filter over $\omega$, for example $\FF=\{\omega\}$.
If ${\mathbb P}({\mathcal F})$ satisfies $\aleph_2$-cc, then $2^{\aleph_0}=\aleph_1$. In particular,
if the iteration of length $\omega_2$ of ${\mathbb P}({\mathcal F})$ with countable supports preserves $\aleph_2$, 
then the continuation of that iteration to have length $\omega_2+1$
does not satisfy $\aleph_2$-cc. 
Similarly for Perfect Set Forcing and Miller forcing. 

\end{corollary}

\begin{proof} Let  ${\mathbb P}_{\omega_2}$ stand for the CS iteration of length $\omega_2$ of Perfect Set Forcing [respectively Miller forcing or Grigorieff forcing with respect to an absolutely definable filter $\FF$]. Suppose that ${\mathbb P}_{\omega_2}$ preserves $\aleph_2$, for example if the forcing was done over a model of CH. Every iterand adds a distinct real, so it is forced by ${\mathbb P}_{\omega_2}$ that
$2^{\aleph_0}\ge\aleph_2$.
Let $\langle \name{p_i}:\,i<\check{\omega}_2\rangle$ be ${\mathbb P}_{\omega_2}$-names for the the first ${\omega_2}$ conditions among the ones constructed as in Lemma \ref{lem:not-too-long} in the forcing $\name{Q}_{\omega_2}$, which is forced by ${\mathbb P}_{\omega_2}$ to be Perfect Set Forcing [respectively Miller forcing or Grigorieff forcing with respect to an absolutely definable filter $\FF$]. Let $q_i\in {\mathbb P}_{\omega_2+1}=
{\mathbb P}_{\omega_2}\ast\name{Q}_{\omega_2}$ be the condition whose support is $\{\omega_2\}$ and such that $q_i(\omega_2)=\name{p_i}$.
Then $\{q_i:\,i<\omega_2\}$ is an antichain of size $\aleph_2$ in ${\mathbb P}_{\omega_2+1}$.
$\eop_{\ref{cor:why-collapse}}$
\end{proof}

\begin{remark} Baumgartner and Laver \cite[Th 5.3]{zbMATH03664937} claimed that a countable support iteration of Perfect Set Forcing of length 
$\omega_2+1$
over a model of CH, collapses $\aleph_2$. Their proof contains a gap, that it is not at all clear why the function $f$ is well defined. In 
\cite[Rem. 0.3]{Goldsterniteration}, Martin Goldstern shows that an iteration with countable supports of length that has cofinality $\omega_1$ of non-trivial forcing, must collapse the ground model continuum to be $\aleph_1$. This shows that the countable support iteration of length $\omega_2+\omega_1$ of
Perfect Set Forcing or any other forcing that adds reals, collapses $\aleph_2$ to $\aleph_1$.
\end{remark}

\begin{corollary}\label{cor:why-collapse} If ${\mathbb P}({\mathcal F}^\ast_\kappa)$ satisfies $\kappa^{++}$-cc, then $2^\kappa=\kappa^+$. In particular,
if the iteration of ${\mathbb P}({\mathcal F}^\ast_\kappa)$ with supports of size $\le\kappa$ over a model of GCH along $\kappa^{++}$ preserves $\kappa^{++}$, then the continuation of that iteration to have the length $\kappa^{++}+1$
does not satisfy $\kappa^{++}$-cc. 
\end{corollary}

\begin{proof} By Theorem \ref{what-added}, an iteration of ${\mathbb P}({\mathcal F}^\ast_\kappa)$ with supports of size $\le\kappa$ over a model of GCH along $\kappa^{++}$ forces $2^\kappa=\kappa^{++}$. If such an iteration does not collapse $\kappa^{++}$, then by Lemma \ref{lem:not-too-long} and a proof similar to that of Corollary \ref{cor:why-collapse}, forcing over the resulting model with one more copy of ${\mathbb P}({\mathcal F}^\ast_\kappa)$ no longer satisfies the
$\kappa^{++}$-cc.
$\eop_{\ref{cor:why-collapse}}$
\end{proof}

\begin{observation}\label{lem:antichains-of-I} If $I$ itself has a set of incomparable elements of size $\kappa^{++}$, then the iteration of 
${\mathbb P}({\mathcal F}^\ast_\kappa)$ along $I$ has an antichain of size $\kappa^{++}$.
\end{observation}

\begin{proof} Suppose that $\{j_\gamma:\,\gamma<\kappa^{++}\}$ is a set of incomparable elements in $I$. Let $p$ be the condition in ${\mathbb P}({\mathcal F}^\ast_\kappa)$ that assigns to every node the set of splitting points equal to $\kappa\setminus\{0\}$. For every $\gamma$ let $\name{p}_\gamma$
be a ${\mathbb P}_{[<j_\gamma]}$-name for $p$. Let $r_\gamma\in {\mathbb P}_{[\le j_\gamma]}$ be the condition whose support is $\{j_\gamma\}$, where it has value $\name{p}_\gamma$. Then $\{r_\gamma:\,\gamma<\kappa^{++}\}$ is an antichain in ${\mathbb P}_I$.
$\eop_{\ref{lem:antichains-of-I}}$
\end{proof}

These considerations imply that if we wish to preserve $\kappa^{++}$ when forcing by ${\mathbb P}_I$, we must restrict our attention to those $I$ that 
avoid the obvious reasons for failing $\kappa^{++}$-cc, as listed above. Hence we must concentrate on the following kind of partial orders. 

\begin{definition}\label{def:kappa-plusplus} We call a partial order $I$ a {\em $\kappa^{++}$-avatar} if it satisfies the following three conditions:
\begin{itemize}
\item $I$ is well-founded,
\item every set of pairwise incomparable elements of $I$ has size $\le\kappa^+$ and
\item every proper segment of $I$ has $\le\kappa^+$ proper initial segments.
$\eod_{\mbox{\tiny{Def.}} \ref{def:kappa-plusplus}}$
\end{itemize}
\end{definition}

\begin{example}\label{ex:avatars} The following are examples of $\kappa^{++}$-avatars:
\begin{enumerate}
\item $\kappa^{++}$,
\item any well-founded order of size $\le\kappa^+$,
\item a disjoint union of any well-founded order of size $\le \kappa^+$ and $\kappa^{++}$ or a lexicographic sum of these.
\end{enumerate}
\end{example}

\begin{observation}\label{obs:avatar} The notion of being a $\kappa^{++}$-avatar is closed under unions of size $\le\kappa^+$.
\end{observation}

A {`}{`}theorem'' we could hence attempt to prove is:

\begin{dtheorem}\label{lem:kappa++} Suppose that GCH holds, $\kappa$ is a regular infinite cardinal, and $I$ is a 
$\kappa^{++}$-avatar.

Then the limit ${\mathbb P}_I$ of the iteration
\[
\langle ({\mathbb P}_J,  \name{{\mathbb P}}({\mathcal F}^\ast)_J):\,J \mbox{ a proper initial segment of }I\rangle
\]
made with supports of size $\le\kappa$
has the $\kappa^{++}$-cc.
\end{dtheorem}

In its generality, the Wishful Theorem is not correct. This follows, setting $\kappa=\aleph_0$ and $I$ to be $\omega_2+\omega_1$, by the already mentioned \cite[Rem. 0.3]{Goldsterniteration}. We explain why many authors have been tempted to prove such theorems. 

Proving Wishful Theorem \ref{lem:kappa++} sounds easy given the litterature. It should be analogous to proving that the countable support iteration of $\omega_2$ copies of Laver forcing 
over a model of CH satisfies the $\aleph_2$-cc. However, despite many citations, Richard Laver's proof of this fact in the original paper 
\cite[Lem. 10]{Laver_Borel} is incorrect, and that for obvious reasons.  Baumgartner's work on Axiom A, \cite[Lem. 7.4]{Ba} claims that over a model of CH any countable iteration of Axiom A forcing of size $\aleph_1$ and of length $\omega_2$ satisfies $\aleph_2$-cc. His proof corrects the obvious mistake from   \cite[Lem. 10]{Laver_Borel}, however, the proof is still incorrect. So are the proofs by Shelah in two editions of his book \cite[III 4.1]{zbMATH03779315} and \cite[III 4.1]{Sh_P} of what has become known as the Shelah Preservation Theorem. It is basically the same theorem and attempted proof as in Baumgartner's article, where the Axiom A property is replaced by its model-theoretic rendition, which is properness. 
This theorem is cited without proof and perhaps jokingly in \cite{GeshckeQuickert}, \cite{Jec86} and \cite{Jech-ST}
(who refers to  \cite{zbMATH05810560}) and it is mentioned under the quotation marks and with no proof in \cite{Goldsterniteration}. Finally, a correct proof was supposedly given by Uri Abraham in \cite{zbMATH05810560}. However, this proof is not correct either, as it makes a mistake in the application of the 
$\Delta$-system Lemma.

In \cite[Th. 2.4]{zbMATH03664937} Baumgartner and Laver claimed that any length iteration of Perfect Set Forcing preserves $\aleph_1$. Their proof is just a sketch. The correct proof follows from the more general result by Kanovei in \cite{zbMATH01335192} and it also follows from \cite{Ba}. 

Baumgartner and Laver
\cite[\S3]{zbMATH03664937} give a careful proof of the $\aleph_2$-chain condition of the CS iteration of length $\omega_2$ of Perfect Set Forcing, over a model of CH. They in fact prove that for any $\alpha<\omega_2$ the iteration of length $\alpha$ has a dense set of size $\aleph_1$. Their proof is specific to Perfect Set Forcing as it uses the notion of being ``$(F,n)$-determined'' that reposes on the fact that the first finitely many levels of a perfect subtree of ${}^{<\omega}2$ form a finite set. In particular this
proof does not apply to Laver forcing or to the Axiom A forcing in general. \footnote{It is therefore clear from reading \cite{zbMATH03664937} that at the time of writing \cite{Ba}, Baumgartner knew that what was to become the ``Shelah's Preservation Theorem'' was wrong. This, combined with the fact that \cite{Ba} does not make a reference to \cite{zbMATH03664937} seems to indicate that putting a wrong iteration theorem for Axiom A in \cite{Ba}, setting the stage for it to be copied in  \cite[III 4.1]{zbMATH03779315} was Baumgartner's way to conclusively end the difference of opinion of who invented what: Baumgartner claimed that he invented Axiom A and that Shelah copied it by dressing it as proper forcing, while Shelah claimed that the two notions were discovered independently. The other existening evidence (see \cite[\S 9]{Property-B}) combined with this legacy of Baumgartner in \cite{Ba}, shows that Baumgartner's version of the events very likely was the correct one. In the light of this, it seems possible that the fact that \cite{zbMATH03664937} does not give a complete argument of the preservation of $\aleph_1$ was a move by both Baumgartner and Laver in the direction of the protection of intellectual rights which the dispute had brought about.} 

As a side remark, while this author did not find any difficulties in understanding those proofs in \cite{zbMATH03664937} that are correct, she is not convinced of the correctness of the claimed simplification of these proofs in \cite{zbMATH01786788}. Particularly by Lemma II.7. which seems to be saying that the iteration of length $\omega_2$ of Perfect Set Forcing is, in modern terminology, strongly proper. Neither is she convinced by the representation of the proof in \cite[pg. 132]{zbMATH04134031} which uses the wrong theorem \cite[Th. 7.4]{Ba} to prove the correct theorem from \cite[\S 3]{zbMATH03664937}.

As for the Laver reals, it is therefore not clear that Laver solved the Borel conjecture the way that he intended it, since his proof that $\aleph_2$ is not collapsed is incorrect. However, as his forcing does make all strong measure zero sets countable, which is contradictory to CH (see the proof by Tomek Bartoszy{\'n}ski and Haim Judah in \cite[\S8.3]{BartoszynskiJudah}), one can conclude that CH no longer holds in his final model and hence $\aleph_2$ is not collapsed. The Borel Conjecture is solved by Laver after all. In our case we do not have the measure-theoretic argument to help conclude that $\kappa^{++}$ is not collapsed, so the situation is not clear. 

Kanovei \cite{zbMATH03685477}, in its own semantics, and the more common presentation of this work in Friedman-Victoria Gitman-Kanovei \cite{zbMATH07077420}, pointed out several difficulties in iterations of sub-forcings of Perfect Set Forcing even for countably many steps.

To conclude, even the linear version iterating along the ordinal $\kappa^{++}$ of Wishful Theorem \ref{lem:kappa++} is quite special when it holds, and sometimes it holds for reasons that fall outside of the theory of iterated forcings- this is the case with Laver forcing.

\section{Concluding Remarks}\label{conclusions} Throughout this paper, we have considered an infinite cardinal $\kappa$  satisfying $\kappa^{<\kappa}=\kappa$
together with the filter ${\mathcal F}^\ast_\kappa$ of its co-bounded subsets. We have shown that the forcing notion ${\mathbb P}(\mathcal F^\ast_\kappa)$ can be iterated with supports of size $\le\kappa$ along
any well-founded order $I$, while preserving cardinals up to and including $\kappa^+$. 

We have also shown that ${\mathbb P}(\mathcal F^\ast_\kappa)$ adds a generic function from $\kappa$ to $\kappa$ that dominates modulo a club every such function from the ground model. As an application, we have provided a proof of a result claimed by Cummings and Shelah
\cite[Th. 1]{zbMATH00819406}. The result states that the structure of ${}^\kappa \kappa$ modulo domination on a club, can embed an arbitrary given
well-founded order, in an extension which preserves cardinals up to and including $\kappa^+$.

We have discussed the difficulties in preserving $\kappa^{++}$, which already arise in contexts much simpler than the one considered here.

\section*{Acknowledgements} The author thanks Vladimir Kanovei for the discussion of Perfect Set Forcing, the geometric iteration and its history. She would also like to acknowledge the clarity and precision of writing in \cite{zbMATH01335192}, thanks to which it was easy to understand that the main idea could apply to the situation treated in this paper.

The paper used the AI product GROK 4.3 by X-AI, which served as a research assistant for rough historical facts, finding references and editing. All mathematical developments are due entirely to the author unassisted by the AI. 

This research would not have been possible without the membership in the libraries of the Bibliothèque Nationale de France and Institut Henri Poincare in Paris and the British Library in London, as well as in the Association of Symbolic Logic (ASL) and Math-Net.Ru, where the latter two provided online access to their journals. The memberships in the Bibliothèque Nationale de France and the Association of Symbolic Logic (ASL) were paid, while those in the British Library, Institut Henri Poincare and Math-Net.Ru were given to me upon a successful application. Therefore I most sincerely thank these organisations.

\bibliography{../bibliomaster}
\bibliographystyle{plainurl}

\end{document}